\documentclass[11pt]{amsart}
\usepackage[english]{babel}
\usepackage[utf8]{inputenc}
\usepackage{fullpage}
\usepackage{enumerate}
\usepackage[toc]{appendix}
\usepackage{amsmath}
\usepackage{amsfonts}
\usepackage{amsthm}
\usepackage{mathtools}
\usepackage{esint}
\usepackage{appendix}
\usepackage{url}

\usepackage{graphics,graphicx}

\usepackage[backref=page]{hyperref}
\hypersetup{
	plainpages=false,
	unicode=false,          
	pdftoolbar=true,        
	pdfmenubar=true,        
	pdffitwindow=true,      
	pdftitle={My title},    
	pdfauthor={Author},     
	pdfsubject={Subject},   
	pdfcreator={Creator},   
	pdfproducer={Producer}, 
	pdfkeywords={keywords}, 
	pdfnewwindow=true,      
	colorlinks=true,       
	linkcolor=blue,          
	citecolor=blue,        
	filecolor=blue,      
	urlcolor=magenta,          
	urlbordercolor=0 1 1
}

\usepackage[hang,center]{subfigure}

\newcommand{\R}{\mathbb R}
\newcommand{\N}{\mathbb N}

\newcommand{\cH}{\mathcal{H}}

\newcommand{\cL}{\mathcal{L}}

\newcommand{\cM}{\mathcal{M}}

\def\xC{{\rm C}}
\def\xLip{ {\rm Lip} }
\def\xlip{ {\rm lip} }

\def\xL{{\rm L}}
\def\xW{{\rm W}}
\def\xdiv{{\rm div}}

\def\xW{{\rm W}}

\newcommand{\G}{{\mathbb G}_{d,n}}
\newcommand{\V}{\|V\|}

\newcommand{\B}{\mathbf{B}}

\newtheorem{theo}{Theorem}[section]
\newtheorem{prop}[theo]{Proposition}
\newtheorem{lemma}[theo]{Lemma}
\newtheorem*{theo*}{Theorem}
\newtheorem{dfn}[theo]{Definition}

\theoremstyle{remark}
\newtheorem{remk}[theo]{Remark}

\DeclareMathOperator*{\supp}{spt}
\DeclareMathOperator*{\dist}{dist}

\DeclareMathOperator*{\tr}{tr}

\usepackage{color}
\definecolor{grey}{rgb}{.7,.7,.7}

\renewcommand{\phi}{\varphi}
\renewcommand{\epsilon}{\varepsilon}
\newcommand{\e}{\epsilon}
\renewcommand{\G}{G_{d,n}}

\title{Weak and approximate curvatures of a measure:\\ a varifold perspective}
\author{Blanche Buet}
\address{Laboratoire de Math\'ematiques d'Orsay, Univ. Paris-Sud, CNRS, Universit\'e Paris-Saclay, F-91405 Orsay, France}
\email{blanche.buet@u-psud.fr}
\author{Gian Paolo Leonardi}
\address{Dipartimento di Matematica, Universit\`a di Trento, Italy}
\email{gianpaolo.leonardi@unitn.it}
\author{Simon Masnou}
\address{Univ Lyon, Universit\'e Claude Bernard Lyon 1, CNRS UMR 5208, Institut Camille Jordan, F-69622 Villeurbanne, France}
\email{masnou@math.univ-lyon1.fr}
\keywords{Varifold, regularized variations, approximate second fundamental form, point cloud}
\thanks{G.P.~Leonardi has been supported by GNAMPA-INdAM. B.~Buet and S.~Masnou acknowledge support from the French National Research Agency (ANR) under grants ANR-12-BS01-0014-01 (project GEOMETRYA) and ANR-14-CE27-001 (project MIRIAM). S.~Masnou acknowledges support from the European Union Horizon 2020 research and innovation programmes under the Marie Skodowska-Curie grant agreement No 777826 NoMADS. Part of this work was also supported by the LABEX MILYON (ANR-10-LABX-0070) of Universit\'e de Lyon, within the program "Investissements d'Avenir" (ANR-11-IDEX- 0007) operated by the French National Research Agency (ANR), by the Department of Physics, Informatics and Mathematics of the University of Modena and Reggio Emilia, and by the JCJC 2018 PEPS of INSMI, CNRS (project: "A unified framework for surface approximation through varifolds").}
\subjclass[2010]{Primary: 49Q15. Secondary: 68U05, 65D18}

\numberwithin{equation}{section}

\begin{document}

\begin{abstract}
By revisiting the notion of generalized second fundamental form originally introduced by Hutchinson for a special class of integral varifolds, we define a \textit{weak curvature tensor} that is particularly well-suited for being extended to general varifolds of any dimension and codimension through regularization. 
The resulting \textit{approximate second fundamental forms} are defined not only for piecewise-smooth surfaces, but also for datasets of very general type (like, e.g., point clouds). We obtain explicitly computable formulas for both weak and approximate curvature tensors, we exhibit structural properties and prove convergence results, and lastly we provide some numerical tests on point clouds that confirm the generality and effectiveness of our approach.
\end{abstract}

     \maketitle


\section*{Introduction}

The aim of this paper is to present a new approach for the computation of second-order extrinsic properties (curvatures) for a very general class of geometric objects, including both smooth or piecewise smooth $d$-submanifolds and discrete datasets in the Euclidean $n$-space. The proposed framework and toolbox rely on a suitable extension of the classical theory of \textit{varifolds}, as we shall describe later on. 

The generality and potentialities of our method can be mostly appreciated on unstructured datasets, like point clouds, for which the determination of curvatures is a very important but extremely delicate task. Indeed, the need of efficient and robust techniques for the analysis of geometric features of general datasets is today of primary importance, due to the variety of both data sources and applications of data analysis, in particular in high dimensions. Among the most common techniques used to measure curvatures of unstructured data, we mention the Moving Least Squares (MLS) algorithm, that is based on the reconstruction of a local, implicit surface from the dataset, from which first- and second-order differential properties can be recovered (see \cite{Amenta2004,levin1998,Yang2007}). Other quite popular methods rely on integral geometry, and in particular on the connection between curvatures and the local expansion of volumes and of covariance-type quantities (see \cite{Coeurjolly2014} and \cite{Merigot2009}). These methods can be used as well for estimating the curvatures of structured data, typically triangle meshes, a problem which has motivated a large number of contributions. An exhaustive description of this literature is far beyond the scope of this paper, the interested reader may refer to~\cite{najman2017modern} and the references therein. In comparison with the various approaches that have been proposed so far for defining and estimating curvatures, we believe that our framework combines two nice properties: it can handle a large category of unsmooth and unstructured data with good consistency and convergence properties, and it is very well suited for numerical purposes. This latter property is a major advantage over other weak notions of curvatures, as those developed in metric measure spaces like Menger's, Wald's or Finsler-Haantjes' curvatures (see~\cite{najman2017modern}) whose numerical approximation is complicated or, at least, computationally expensive.

Our framework is based on the notion of varifold (see below), a classical tool in geometric measure theory. Surprisingly, the representation and analysis of discrete surfaces has rarely taken advantage of the tools and techniques developed in geometric measure theory. A remarkable exception is the theory of normal cycles (a generalization of unit normal bundles \cite{Wintgen1982,fu1993,zahle}) that has been recently adapted to the reconstruction of curvatures from offsets of distance-like functions associated with datasets (see \cite{thibert,morvan_cohen_steiner,morvan_book}).

\subsection*{The theory of varifolds} Varifolds represent very natural generalizations of classical $d$-surfaces, as they encode, loosely speaking, a joint distribution of mass and tangents. More technically, varifolds are Radon measures defined on the Grassmann bundle $\R^n\times \G$ whose elements are pairs $(x,S)$ specifying a position in space and an unoriented $d$-plane. Varifolds have been proposed more than 50 years ago by Almgren \cite{Almgren} as a mathematical model for soap films, bubble clusters, crystals, and grain boundaries. After Allard's fundamental work \cite{Allard72}, they have been successfully used in the context of Geometric Measure Theory, Geometric Analysis, and Calculus of Variations. Among those, Almgren--Pitts min--max theory (see \cite{DeLellisColding} for a survey) has fundamental consequences, from existence of smooth embedded minimal hypersurfaces in a given compact Riemannian manifold \cite{Pitts} to the recent proof of Willmore's conjecture \cite{CodaNeves}. Another successful application of varifolds resulted in the definition and study of a general weak mean curvature flow in \cite{brakke}, which  allowed to prove existence of mean curvature evolution with singularities in \cite{TonegawaKim}. Beyond the theory of rectifiable varifolds, the flexibility of the varifold structure has proven to be relevant to model diffuse interfaces, e.g., phase field approximations, and took for instance a crucial part in the proof of the convergence of the Allen-Cahn equation to Brakke's mean curvature flow~\cite{Ilmanen, Tone, TakaTone}, or in the proof of the $\Gamma$--convergence of Cahn-Hilliard type energies to the Willmore energy (up to an additional perimeter term)~\cite{Roger}. In all these contributions a key element is the possibility to use second-order properties of varifolds associated with either nicely diffuse or fairly well concentrated {(rectifiable)} measures. The notion of approximate curvature tensors proposed in our paper opens new perspectives for calculating second-order properties associated with \textit{any} varifold.

\par
Within the classical theory of varifolds, a generalized notion of curvature is encoded in the \textit{first variation} operator, see \cite{Allard72,Allard75,simon}.
The first variation $\delta V$ of a varifold $V$ is a vector--valued distribution of order $1$, i.e., a linear and continuous functional defined on compactly supported vector fields of class $C^1$. Given a vector field $Y$, the first variation of $V$ applied to $Y$ equals the derivative of the mass of the varifold obtained by transforming (push-forwarding) $V$ through the one-parameter flow generated by $Y$ (see Section \ref{section:pre}). Note that such a derivative, in the case of a smooth $d$-surface $M$ without boundary, can be obtained by integrating (minus) the scalar product between the mean curvature vector $H^M$ and the vector field $Y$ on $M$. 
The so--called \textit{Allard varifolds}, i.e.,  varifolds whose first variation is a Radon measure, constitute a very important class of measures. In particular, the \textit{generalized mean curvature} $H^V$ of an Allard varifold $V$ can be defined as the vector--valued density of $\delta V$ with respect to the mass (or weight) measure $\V$. We mention that, as a key ingredient of the \textit{direct method} of the Calculus of Variations, a regularity theory is available for varifolds with generalized mean curvature in $L^p$ for $p>d$ (see \cite{Allard72}). With less integrability, weak regularity results still hold: all integral varifolds with locally bounded first variation are uniformly rectifiable~\cite{Luckhaus} and $C^2$-rectifiable~\cite{Menne2013}. Let us also mention the extension of Allard's regularity theory to the anisotropic setting recently achieved in~\cite{DePhilippis}. 

More than a decade after Allard's work, a theory of \textit{curvature varifolds} was proposed by Hutchinson \cite{Hutchinson,Hutchinson2}. This theory provides a notion of \textit{generalized second fundamental form}, together with existence and regularity results for solutions to variational problems involving curvature--dependent functionals. 
As a matter of fact, if Almgren-Allard's theory of varifolds has been mostly known and used by specialists of Geometric Measure Theory, Geometric Analysis, and Calculus of Variations, it seems that an even smaller community of mathematicians has been aware of Hutchinson's theory of curvature varifolds. Among the most noticeable works that refer to this theory, or make some use of it, we recall the extension to curvature varifolds with boundary~\cite{Mantegazza} and the recent reformulation and extension~\cite{menne,MenneSharrer2018}, plus several works on the minimization of curvature-dependent energies \cite{bellettini1997variational, Mondino, kuwert2014existence}, on the relaxation of elastica and Willmore functionals \cite{bellettini2010approximation,bellettini2007varifolds}, on some applications to mechanics \cite{giaquinta2009currents}, or on geometric flows arising from cosmology \cite{huisken2001inverse}. Other closely related works are \cite{anzellotti1990curvatures, delladio2000differential, delladio1997special, delladio1995oriented} on the theory of generalized Gauss graphs, and \cite{ambrosio1998approximation} on a comparison between curvature varifolds and measures arising as limits of suitable classes of multiple--valued functions. 
 
As we already pointed out in our previous work \cite{BuetLeonardiMasnou}, the theory of varifolds has seen no substantial applications in the fields of applied mathematics. A possible explanation is that some key tools, like the first variation operator, are not directly applicable to general varifolds, and in particular to varifolds arising from discrete datasets. For example, the first variation of a point cloud varifold is not a measure, but only a distribution resulting from a directional / tangential derivative of a finite sum of weighted Dirac's deltas (see Section \ref{section:varifolds}). Therefore, the standard first variation of a point cloud varifold does not directly provide any consistent notion of (mean) curvature for that kind of dataset. 

\subsection*{Weak and approximate curvature tensors}
In \cite{BuetLeonardiMasnou} we showed how to define consistent notions of approximate mean curvature for any varifold, including those of discrete type, and we proved a series of results that opened the way for a systematic application of the (extended) theory of varifolds in the context of discrete and computational geometry. Here, our main objective is to push forward our previous work in the direction of a general theory of approximate curvature tensors, to be made available for the whole class of varifold measures.

In respect of this objective, a very natural idea would be to apply the same regularization scheme as in \cite{BuetLeonardiMasnou} to Hutchinson's generalized curvature tensor. However, this is not possible without a suitable revision of Hutchinson's theory itself. The fact that Hutchinson's generalized curvature tensor is not defined as a distribution of order $1$ (see formula \eqref{eqHutchinsonBPI}) is the main, structural obstruction to a direct application of our regularization technique. By taking a closer look to the definition of curvature tensor given by Hutchinson (see also Section~\ref{sectionHutchinsonSecondFundamentalForm}), one can easily realize that the constraints imposed on the curvature tensor by general test functions $\phi(x,S)$ -- specifically, by those that are nonlinear with respect to the Grassmannian variable $S$ -- make Hutchinson's notion extremely rigid, in the sense that every blow-up of an integral varifold admitting generalized curvatures in $L^p_{loc}(\V)$, with $p>d$, necessarily consists of a finite union of $d$-planes with multiplicities.

In order to overcome the structural obstruction mentioned above, we need to modify Hutchinson's definition of generalized curvature tensor in order to obtain less rigid, but at the same time fully consistent notions of weak curvatures. To this aim, we introduce in Section \ref{sectionGeneralizedCurvatures} a suitable family of variation operators (the \textit{$G$-linear variations}, see Definition~\ref{dfnVariations}) and obtain from them a \textit{weak second fundamental form} (WSFF, see Definition \ref{dfn:weakFundForm}). The path we follow is similar to the standard one that, starting from the first variation $\delta V$, leads to the generalized mean curvature $H^V$. Our WSFF can be considered in some sense a \textit{core distributional notion} extracted from Hutchinson's original definition. We refer the interested reader to Section \ref{sectionHutchinsonSecondFundamentalForm} for a thorough comparison between Hutchinson's tensor and our WSFF. As explained in Section \ref{sectionGeneralizedCurvatures}, a net advantage of the WSFF with respect to Hutchinson's tensor is that it only depends on the space variable $x$ and can be directly computed from the $G$-linear variations through explicit formulae. As far as we know, these formulae are new and provide the basis for the definition of explicitly computable \textit{approximate second fundamental forms}. 

Let us note that the mean curvature vector is orthogonal to the tangent plane for a smooth submanifold as well as for an integral varifold, at least almost everywhere (\cite{brakke}) and consequently there is some margin in the definition of a weak second fundamental form consistent for regular varifolds. This observation allows to define a consistent variant of the WSFF, the orthogonal weak second fundamental form denoted by WSFF$^\perp$ (see Definition~\ref{dfnBetaAijkPerp}), that both enforces the structural properties \eqref{eqStructuralProperties} (which are satisfied by the second fundamental form of any smooth $d$-surface) and gives a better rate of convergence after regularization (see Theorem~\ref{theoConvergence2}).

In Section \ref{Section:RegWSFF} we apply the regularization technique of \cite{BuetLeonardiMasnou} to the WSFF (as well as to its variant WSFF$^\perp$) and obtain corresponding notions of approximate second fundamental form for general varifolds ($\e$-WSFF or $\e$-WSFF$^\perp$, where $\e>0$ is a parameter that controls the approximation scale, see Definitions \ref{propDfnRegularizedSecondFundForm} and \ref{dfnBetaAijkPerpReg}). The convergence properties of these approximate curvature tensors are then analyzed in the last part of the section, resulting in Theorem~\ref{theoConvergence1} and \ref{theoConvergence2}. We insist that those results are \emph{convergence} and not only \emph{consistency} results, in the sense that they compare the approximate second fundamental forms of a sequence of varifolds with the second fundamental form of its weak-$\ast$ limit.

Finally, in Section \ref{sectionNumerics} we show some numerical results on point clouds that illustrate the robustness of the method. In particular, some examples of $2$-dimensional point clouds in $\R^3$ are considered. For them, we compute the approximate principal curvatures also in presence of noise and singularities. 

In conclusion, the approximate second fundamental form $\e$-WSFF (as well as its orthogonal variant $\e$-WSFF$^\perp$) constitutes a general, robust, and easily-computable tool for extracting local curvature information from very general geometric objects. Several possible applications as well as future research directions can be thus envisaged. Some of them are (1) the development of a theory of discrete geometric evolutions (e.g., the discrete mean curvature flow) for point clouds with application to either shape smoothing and segmentation, or to the study of interfaces in systems of microscopic particles (see the context in~\cite{Spohn1993}); (2) the use of curvature estimators for data analysis and classification; (3) the determination of intrinsic properties of unstructured datasets (like the density of a point cloud, or the second-order derivatives of functions in grid-free numerical methods) via curvature--related properties.

\section{Preliminaries}
\label{section:pre} 
\subsection*{Notations} 
In what follows, $\N$ and $\R$ denote, respectively, the set of natural and of real numbers. Given $n\in \N$, $n\ge 1$, $x\in \R^{n}$ and $r>0$, we denote by $(e_1, \ldots, e_n)$ the canonical basis of $\R^n$, $|x|$ the Euclidean norm of $x$ and set $B_r(x) = \left\lbrace y\in \R^{n}:\ |y-x| < r \right\rbrace$. If $B$ is an open set, writing $A\subset\subset B$ means that $A$ is a relatively compact subset of $B$. We define the $\delta$-tubular neighborhood of a set $A\subset \R^n$ as 
\[
A^\delta = \bigcup_{x \in A} B_\delta (x) = \{ y \in \R^n \, | \, d( y,A) <\delta \}.
\]
$\cL^n$ denotes the $n$--dimensional Lebesgue measure and $\omega_n = \cL^n (B_1(0))$. 
Given a metric space $(X,\delta)$ and a function $f:X\to \R$, we denote by $\xlip(f)$ the Lipschitz constant of $f$. Then, $\xLip_L (X)$ denotes the space of real--valued Lipschitz functions $f$ defined on $X$ and such that $\xlip(f)\leq L$. 
$\cM_{loc} (X)^m$ is the space of $\R^m$--valued Radon measures and $\cM (X)^m$ is the space of $\R^m$--valued finite Radon measures on the metric space $(X,\delta)$. We denote by $| \mu |$ the total variation of a measure $\mu$.

From now on, we fix $d, \, n \in \N$ with $1\leq d \leq n$. By $\Omega \subset \R^n$ we shall always denote an open set. 
The $d$--dimensional Hausdorff measure in $\R^{n}$ is denoted by $\cH^d$. We let $\G$ be the Grassmannian manifold of $d$-dimensional vector subspaces of $\R^n$. We recall that a  $d$-dimensional subspace $T$ of $\R^n$ is equivalently represented by the orthogonal projection onto $T$, denoted as $\Pi_T$ (or simply $T$ when there is no possible confusion). $\G$ is equipped with the metric $d(T,P) = \Vert \Pi_T - \Pi_P \Vert$, where $\Vert \cdot \Vert$ denotes the operator norm on the space $L(\R^{n};\R^{n})$ of linear endomorphisms of $\R^{n}$. We will usually write $\Vert T-P\|$ instead of $\Vert \Pi_T - \Pi_P \Vert$. Other norms on vectors and matrices will be used in the sequel, like the $\ell_\infty$ norm defined as $|v|_\infty = \max_j |v_j|$ when $v$ is a vector of $\R^n$, or $|M|_\infty = \max_{i,j} |M_{ij}|$ if $M$ is a matrix with real coefficients.
 
Given a continuous $\R^m$--valued function $f$ defined in $\Omega$, its support $\supp f$ is the closure in $\Omega$ of $\{ y \in \Omega \: | \: f(y) \neq 0 \}$. 
$\xC_o^0 (\Omega)$ is the closure of $\xC_c^0(\Omega)$ in $\xC^{0}(\Omega)$ with respect to the norm $\|u \|_{\infty} = \sup_{x\in \Omega}|u(x)|$. 
Given $k \in \N$, $\xC_c^k (\Omega)$ is the space of real--valued functions of class $\xC^k$ with compact support in $\Omega$. 
$\xC^{0,1}(\R)$ and $\xC^{1,1}(\R)$ denote, respectively, the space of Lipschitz functions and the space of functions of class $C^{1}$ with Lipschitz derivative on $\R$. On such spaces we shall consider the norms $\|u\|_{1,\infty} = \|u\|_{\infty}+\xlip(u)$ and $\|u\|_{2,\infty} = \|u\|_{\infty}+\|u'\|_{1,\infty}$, respectively.

Given $f \in \xC^1(\Omega)$, $X \in \xC^1(\Omega, \R^n)$ and $S \in \G$, we define for $x \in \Omega$, the $S$--gradient of $f$ at $x$ as well as the $S$--divergence of $X$ at $x$ as
\[
\nabla^S f(x) = \Pi_S \nabla f(x)  \quad \text{and} \quad \left\lbrace 
\begin{array}{l}
{\xdiv}_S X(x) = \sum_{i=1}^n DX(x) \tau_i \cdot \tau_i, \\\
(\tau_1, \ldots, \tau_d) \text{ orthonormal basis of } S \in \G
\end{array} \right. .
\]
Notice that if $M \subset \Omega$ is a $d$--submanifold, those notations are shortened in the classical ones $\nabla^M f(x) = \nabla^{T_x M} f(x)$ and ${\xdiv}_M X(x) =  {\xdiv}_{T_x M} X(x)$.
%
\subsection*{Second fundamental form: the classical case}
\label{sectionClassicalSecondFundamentalForm}

For the sake of clarity, and in order to fix further notations that will be used in the next sections, we recall the essential definitions and facts about the classical second fundamental form. In view of the generalized notions that will be introduced after, we closely follow the notation of \cite{Hutchinson}. We consider a smooth $d$-dimensional submanifold $M$ of $\R^n$ with the standard induced metric. Given $x\in M$, we denote by $P(x)$ the orthogonal projection onto the tangent space $T_xM$; such a projection is represented by the matrix $P_{ij}(x)$ with respect to the standard basis of $\R^n$. The usual covariant derivative in $\R^n$ is denoted by $D$. Assuming $x\in M$ fixed, and given a vector $V\in T_x\R^n = \R^n$, we let $V^T = P(x)\,V$ and $V^\perp = V - V^T$. 

We denote by, respectively, $TM$ and $(TM)^\perp$ the tangential and the normal bundle associated with $M$, so that we have the splitting $TM \oplus (TM)^\perp = T\R^n$. 

We can now introduce the classical, second fundamental form of $M$, as the bilinear and symmetric map $II: TM\times TM \to TM^\perp$ defined as
\[
II(U,V) = (D_{U}V)^\perp\,.
\]
For technical reasons it is convenient to consider the \textit{extended second fundamental form} of $M$, which is defined as $\B(U,V)= II(U^T,V^T)$ for all $U,V\in T\R^n$. Therefore we have that the map $B:T\R^n\times T\R^n\to T\R^n$ fully encodes the second fundamental form. We set 
\begin{equation}\label{Bijk}
B_{ij}^k =  \B(e_i,e_j) \cdot e_k \,,
\end{equation}
where $\{e_i:\ i=1,\dots,n\}$ is the canonical basis of $T\R^n\simeq \R^n$. It is easy to check that the coefficient set $\{B_{ij}^k:\ i,j,k=1,\dots,n\}$ uniquely identifies $\B$. 

We now present an equivalent way of defining the extended second fundamental form by computing tangential derivatives of the orthogonal projection $P(x)$ onto the tangent space to $M$ at $x$. More precisely, let us set 
\begin{equation}\label{Aijk}
A_{ijk}(x) =  \nabla^M P_{jk}(x) \cdot e_i
\end{equation} 
whenever $x\in M$ and $i,j,k=1,\dots,n$. It is not difficult to check that
\begin{equation}\label{eqFromBtoA}
A_{ijk} = B_{ij}^k + B_{ik}^j
\end{equation}
and, reciprocally,
\begin{equation} \label{eqFromAtoB}
B_{ij}^k = \frac 12 \left(A_{ijk} + A_{jik} - A_{kij}\right)\,,
\end{equation}
see also Proposition \ref{propWellDefined} for a proof of these identities. We note for future reference the symmetry properties $A_{ijk} = A_{ikj}$ and $B_{ij}^k = B_{ji}^k$. The symmetry of $A_{ijk}$ follows from $P_{jk} = P_{kj}$. The symmetry of $B_{ij}^k$ relies upon the well-known identity $(D_U V)^\perp = (D_V U)^\perp$, which follows from the fact that the Levi-Civita connection $D$ is torsion--free.

Let us now recall the classical divergence theorem on $M$. Given a smooth vector field $X$ with compact support on an open neighborhood of $M$, we have 
\begin{equation}\label{eq:divthmonmanifold}
\int_M \xdiv_M (PX) \, d\cH^d = \int_{\partial M} X\cdot \eta\, d\cH^{d-1}\,,
\end{equation}
where $\eta$ is the outward pointing conormal to $\partial M$. 

Let us now consider vector fields of the form $X = X_{ijk} (x) = \phi(x) P_{jk}(x) e_i$, with $\phi\in \xC_c^1(\R^n)$. Recalling that $A_{ijk}(x)  = \nabla^M P_{jk}(x) \cdot e_i$, from \eqref{eq:divthmonmanifold} we obtain the identity
\begin{align} \label{eqHutchinsonLinearP_2}
\int_M P_{jk} \nabla^M \phi \cdot e_i\, d\cH^d = -\int_M \phi\Big(A_{ijk}+ P_{jk} \sum_q  A_{qiq}   \Big)\, d\cH^d  + \int_{\partial M} \phi P_{jk} (\eta \cdot e_i)\, d\cH^{d-1}\: ,
\end{align}
for every $i, j, k = 1, \ldots, n$. It is worth noticing that 
\[
\sum_q  A_{qiq}(x) = H(x) \cdot e_i\,,
\]
where $H(x)$ denotes the mean curvature vector of $M$ at $x$. We shall see in Section \ref{sectionGeneralizedCurvatures} that formula \eqref{eqHutchinsonLinearP_2} can be used to define the curvature tensor $A = \{A_{ijk}\}$ in a weak sense.

\section{Varifolds}
\label{section:varifolds}

Here we briefly recall the main definitions and some basic facts about varifolds, and refer the interested reader to \cite{simon} for more details. 

\begin{dfn}[General $d$--varifold]\label{def:gdv}\rm 
	Let $\Omega \subset \R^n$ be an open set. A $d$--varifold in $\Omega$ is a non-negative Radon measure on $\Omega \times G_{d,n}$. 
\end{dfn}

An important class of varifolds is represented by the so-called \emph{rectifiable} varifolds.
\begin{dfn}[Rectifiable $d$--varifold]\label{def:rdv}\rm
	Given an open set $\Omega \subset \R^n$, let $M$ be a countably $d$--rectifiable set and $\theta$ be a non negative function with $\theta > 0$ $\cH^d$--almost everywhere in $M$. A rectifiable $d$--varifold $V= v(M,\theta)$ in $\Omega$ is a non-negative Radon measure on $\Omega \times G_{d,n}$ of the form $V= \theta \mathcal{H}^d_{| M} \otimes \delta_{T_x M}$ i.e.
	\[
	\int_{\Omega \times G_{d,n}} \varphi (x,T) \, dV(x,T) = \int_M \varphi (x, T_x M) \, \theta(x) \, d \mathcal{H}^d (x) \quad \forall \varphi \in \xC_c^0 (\Omega \times G_{d,n} , \mathbb{R})
	\] where $T_x M$ is the approximate tangent space at $x$ which exists $\mathcal{H}^d$--almost everywhere in $M$. The function $\theta$ is called the \emph{multiplicity} of the rectifiable varifold. If additionally $\theta(x)\in \N$ for $\mathcal{H}^d$--almost every $x\in M$, we say that $V$ is an \emph{integral} varifold.
\end{dfn}

In \cite{BuetLeonardiMasnou} we considered \textit{discrete varifolds} (i.e., varifolds that are defined by a finite set of real parameters) as a relevant class of varifolds associated with discrete geometric data. Among them, an important subclass is that of \textit{point cloud varifolds}.
\begin{dfn}[Point cloud varifold]\label{def:PCV}\rm
	Let $\{ (x_i, P_i, m_i) \}_{i=1 \ldots N} \subset \R^n$ be a finite set of triplets, where $x_i\in\R^n$, $P_i\in \G$, and $m_i\in (0,+\infty)$ for all $i$. We associate with this set of triplets the \emph{point cloud $d$--varifold}
	\[
	V^{pc} = \sum_{i=1}^N m_i \, \delta_{x_i} \otimes \delta_{P_i}\,.
	\]
\end{dfn}
Note that a point cloud varifold is a $d$--varifold even though it is not $d$--rectifiable as its support is zero-dimensional.

\begin{dfn}[Mass]\rm
	The mass of a general varifold $V$ is the positive Radon measure defined by $\Vert V \Vert (B) = V (\pi^{-1} (B))$ for every $B \subset \Omega$ Borel, with $\pi:\Omega\times G_{d,n}\to \Omega$ defined by $\pi(x,S) = x$. 
	For example, the mass of a $d$--rectifiable varifold $V=v(M,\theta)$ is the measure $\V=\theta \mathcal{H}^d_{| M}$, while the mass of a point cloud $d$--varifold is the measure $\|V^{pc}\| = \sum_{i=1}^N m_i \delta_{x_i}$.
\end{dfn}
	
The following result is proved via a standard disintegration of the measure $V$ in terms of its mass $\V$ (see for instance \cite{ambrosio}).
\begin{prop}[Young-measure representation]\label{prop:Young} Given a $d$--varifold $V$ on $\Omega$, there exists a family of probability measures $\{\nu_{x}\}_{x}$ on $\G$ defined for $\V$-almost all $x\in \Omega$,  such that $V = \V\otimes \{\nu_{x}\}_{x}$, that is,
	\[
	V(\phi) = \int_{x\in\Omega}\int_{S\in\G}\phi(x,S)\, d\nu_{x}(S)\, d\V(x)
	\]
	for all $\phi\in \xC^{0}_{c}(\Omega\times\G)$.
\end{prop}

We recall that a sequence $(\mu_i)_i$ of Radon measures defined on a locally compact metric space is said to weakly--$\ast$ converge to a Radon measure $\mu$ (in symbols, $\mu_i\xrightharpoonup[]{\: \ast \:}\mu$) if, for every $\phi\in C^0_c(\Omega)$, $\mu_i(\phi)\to \mu(\phi)$ as $i\to\infty$. 
\begin{dfn}[Convergence of varifolds]\rm
	A sequence of $d$--varifolds $(V_i)_i$ weakly--$\ast$ converges to a $d$--varifold $V$ in $\Omega$ if, for all $\varphi \in \xC_c^0 (\Omega \times G_{d,n})$,
	\[
	\langle V_i,\phi\rangle=\int_{\Omega \times \G} \phi (x,P) \, dV_i (x,P) \xrightarrow[i \to \infty]{} \langle V,\phi\rangle= \int_{\Omega \times \G} \phi(x,P) \, dV(x,P) \: .
	\]
\end{dfn}

We now recall the definition of \emph{Bounded Lipschitz distance} between two Radon measures.  It is also called \emph{flat metric} and can be seen as a modified $1$--Wasserstein distance which allows the comparison of measures with different masses (see \cite{villani_book}, \cite{piccoli_rossi}). In contrast, the $1$--Wasserstein distance between two measures with different masses is infinite.

\begin{dfn}[Bounded Lipschitz distance] \label{dfn_flat_distance}\rm 
	Being $\mu$ and $\nu$ two Radon measures on a locally compact metric space $(X,d)$, we define
	\[
	\Delta (\mu,\nu) = \sup \left\lbrace \left| \int_X \phi \, d\mu - \int_X \phi \, d\nu \right| \: : \:\phi \in \xLip_1 (X), \:  \| \phi \|_{\infty} \leq 1 \right\rbrace\,.
	\]
	It is well-known that $\Delta(\mu,\nu)$ defines a distance on the space of Radon measures on $X$, called the \emph{Bounded Lipschitz distance}.
\end{dfn}

Hereafter we introduce some special notation for the $\Delta$ distance between varifolds.
\begin{dfn} \label{dfn_localized_bounded_lipschitz_distance}\rm 
	Let $\Omega \subset \R^n$ be an open set and let $V,W$ be two $d$--varifolds on $\Omega$. For any open set $U\subset \Omega$ we define 
	\[
	\Delta_U (V,W) = \sup \left\lbrace \left| \int_{\Omega \times \G} \phi \, dV - \int_{\Omega \times \G} \phi \, dW \right| \: : \: \begin{array}{l}
	\phi \in \xLip_1 (\Omega \times \G), \:  \| \phi \|_{\infty} \leq 1 \\
	\text{ and } \supp \phi \subset U \times \G \end{array} \right\rbrace 
	\]
	and
	\[
	\Delta_U(\| V \|, \| W \|) = \sup \left\lbrace \left| \int_{\Omega } \phi \, d \| V \| - \int_{\Omega } \phi \, d \| W \| \right| \: : \: \begin{array}{l}
	\phi \in \xLip_1 (\Omega), \:  \| \phi \|_{\infty} \leq 1 \\
	\text{ and } \supp \phi \subset U  \end{array} \right\rbrace \: . 
	\]
	We shall often drop the subscript when $U  = \Omega$, that is we set 
	\[
	\Delta(V,W) = \Delta_{\Omega}(V,W)\quad\text{and}\quad \Delta(\|V\|,\|W\|) = \Delta_{\Omega}(\|V\|,\|W\|)\,,
	\] 
	thus making the dependence upon the domain implicit whenever this does not create any confusion.
\end{dfn}

The following fact is well-known (see \cite{villani_book,bogachev2007}).
\begin{prop}\label{BLDversusWeakstar}
	Let $\mu,(\mu_{i})_{i}$, $i\in \N$, be Radon measures on a locally compact and separable metric space $(X,\delta)$. Assume that $\mu(X) + \sup_{i}\mu_{i}(X) <+\infty$ and that there exists a compact set $K\subset X$ such that the supports of $\mu$ and of $\mu_{i}$ are contained in $K$ for all $i\in \N$. Then $\mu_{i}\xrightharpoonup[]{\: \ast \:}\mu$ if and only if $\Delta(\mu_{i},\mu)\to 0$ as $i\to\infty$.
\end{prop}

Let us eventually introduce the following linear form that encodes a generalized notion of mean curvature.
\begin{dfn}[First variation of a varifold, \cite{Allard72}] The first variation of a $d$--varifold in $\Omega \subset \R^n$ is the vector--valued distribution (of order $1$) defined for any vector field $X\in {\xC}_c^1 (\Omega , \R^n )$ as
	\[
	\delta V(X) = \int_{\Omega \times G_{d,n}} {\xdiv}_S X (x) \, d V (x,S)\,.
	\]
\end{dfn}
\begin{remk}\label{remk:morefirstvar}\rm
	It is convenient to define the action of $\delta V$ on a function $\phi\in C^{1}_{c}(\Omega)$ as the vector
	\[
	\delta V(\phi) = \big(\delta V(\phi\, e_{1}),\dots,\delta V(\phi\, e_{n})\big)   = \int_{\Omega\times \G} \nabla^S \phi(x)\, dV(x,S)\,.
	\]
	We also notice that $\delta V(X)$ is well-defined whenever $X$ is a Lipschitz vector field such that the measure $\V$ of the set of non-differentiability points for $X$ is zero.  
\end{remk}

\noindent
The definition of first variation can be motivated as follows. Let $\phi^{X}_{t}$ be the one-parameter group of diffeomorphisms generated by the flow of the vector field $X$. Let $\Phi^{X}_{t}$ be the mapping defined on $\R^{n}\times\G$ as
\[
\Phi^{X}_{t}(x,S) = (\phi^{X}_{t}(x), d\phi^{X}_{t}(S))\,.
\]
Set $V_{t}$ as the push--forward of the varifold measure $V$ by the mapping $\Phi^{X}_{t}$, that is,
\[
V_t(\phi) = \int\phi(\Phi^{X}_{t}(x,S))\, J^S\phi^{X}_{t}(x)\, dV(x,S)\,,
\]
where $J^S\phi^{X}_{t}(x)$ denotes the tangential Jacobian of $\phi^{X}_{t}$ at $x$. Then, assuming $\supp(X) \subset \subset A$ for some relatively compact open set $A\subset \Omega$, one has the identity
\[
\delta V(X) = \frac{d}{dt} \|V_{t}\|(A)_{|t=0}\,.
\]

\noindent The linear functional $\delta V$ is, by definition, continuous with respect to the $C^{1}$-topology on $C^{1}_{c}(\Omega,\R^{n})$, however in general it is not continuous with respect to the $\xC_c^0$ topology. In the special case when this is satisfied, that is, for any fixed compact set $K\subset\Omega$ there exists a constant $c_{K}>0$, such that for any vector field $X\in \xC_c^1 (\Omega , \R^n)$ with $\supp X \subset K$, one has
\[
| \delta V (X) | \leq c_K \sup_K |X| \: ,
\] 
we say that $V$ has a \emph{locally bounded first variation}. In this case, by Riesz Theorem, there exists a vector--valued Radon measure on $\Omega$ (still denoted as $\delta V$) such that
\[
\delta V (X) = \int_\Omega X \cdot d\,\delta V \quad \text{for every } X \in \xC_c^0 (\Omega,\mathbb{R}^n)
\]
Thanks to Radon-Nikodym Theorem, we can decompose $\delta V$ as 
\begin{equation}\label{RNLdecomp}
\delta V = - H \|V\| + \delta V_s \: ,
\end{equation}
where $H \in \left( \xL^1_{loc}(\Omega, \| V \|) \right)^n$ and $\delta V_s$ is singular with respect to $\V$. The function $H$ is called the {\it generalized mean curvature vector}. By the divergence theorem, $H$ coincides with the classical mean curvature vector if $V=v(M,1)$, where $M$ is a $d$-dimensional submanifold of class $\xC^2$. 
\medskip

\section{Weak Second Fundamental Form}
\label{sectionGeneralizedCurvatures}

The left-hand side of \eqref{eqHutchinsonLinearP_2} motivates the following definition. 
\begin{dfn}[$G$-linear variation] \label{dfnVariations}
Let $\Omega \subset \R^n$ be an open set and let $V$ be a $d$--varifold in $\Omega$. We fix $i,j,k \in \{1, \ldots, n\}$ and define the distribution $\delta_{ijk} V: \xC_c^1 (\Omega) \rightarrow \R$ as
\[
\delta_{ijk}V(\phi) = \int_{\Omega \times \G} S_{jk} \,  \nabla^S \phi (y) \cdot e_i \, dV(y,S) \: .
\]
We say that $\delta_{ijk} V$ is a $G$-linear variation (or, shortly, a variation) of $V$.
\end{dfn}

\begin{dfn}[Bounded variations] \label{dfnBoundedVariations}
Let $\Omega \subset \R^n$ be an open set and let $V$ be a $d$--varifold in $\Omega$. We say that $V$ has locally \emph{bounded variations} if and only if for $i,j,k = 1 \ldots n$, $\delta_{ijk} V$ is a Radon measure. In this case there exist $\beta_{ijk} \in \xL_{loc}^1(\V)$ and Radon measures $\left( \delta_{ijk} V \right)_s$ that are singular w.r.t. $\V$, such that
\begin{equation}\label{eq:deltaijk-decomposed}
\delta_{ijk} V = - \beta_{ijk} \, \V + \left( \delta_{ijk} V \right)_{s} \: .
\end{equation}
\end{dfn}
The proof of the following proposition is immediate.
\begin{prop} 
For any $S\in \G$ we have $\tr(S) = d$ and therefore 
\[
\sum_{j=1}^n   \delta_{ijj} V  = d \, \delta V \cdot e_i \: ,\qquad \forall\, i = 1, \ldots, n\,.
\]
Hence, if $V$ has locally bounded variations then in particular it has locally bounded \textit{first} variation. 
\end{prop}

When $M$ is a compact $d$--submanifold of class $\xC^2$, and $V = v(M,1)$ is the associated varifold with multiplicity $1$, then we infer from \eqref{eqHutchinsonLinearP_2} and \eqref{eq:deltaijk-decomposed} that
\begin{equation} \label{eqBetaAijk}
\beta_{ijk} (x) = A_{ijk} (x) + P_{jk}(x)\sum_q A_{qiq} (x) \quad \text{and} \quad \left( \delta_{ijk} V \right)_{s} = P_{jk} \eta_i \cH^{d-1}_{| \partial M}\,,
\end{equation}
for all $x\in M$. Note that in this case we have
\[
P_{jk}(x) = \int_{\G} S_{jk} \, d \nu_x (S)\,,
\]
where $\nu_x = \delta_{P(x)}$. In Lemma \ref{lemma:systemeInversible2} below we prove some crucial properties of the linear system 
\begin{equation} \label{eqBetaAijk-2}
A_{ijk} (x) + \int_{\G} S_{jk} \, d \nu_x (S)\sum_q A_{qiq} (x) = \beta_{ijk} (x)\,,\qquad i,j,k=1,\ldots,n\,,
\end{equation}
and in particular the fact that it completely characterizes the tensor $A=\{A_{ijk}(x)\}$.
\begin{lemma} \label{lemma:systemeInversible2}
Let $c$ be a $n \times n$ positive semi-definite symmetric matrix and let $b = (b_{ijk}) \in \R^{n^3}$. Let us consider the set of $n^3$ equations of unknowns $(a_{ijk})_{i,j,k=1 \ldots n}$
\begin{equation} \label{eq:mainSystem2}
a_{ijk} + c_{jk} \sum_{q} a_{qiq} = b_{ijk} \: , \quad  \text{for } i,j,k=1 \ldots n \: ,
\end{equation}
and let $L$ be the $n^3\times n^3$ matrix associated with system \eqref{eq:mainSystem2}. Then,
\begin{enumerate}
\item the matrix $L$ is invertible and $\det (L) = det (I_n + c)$,
\item the unique solution of the system is
\begin{equation}\label{aijkexplicit}
a_{ijk} = b_{ijk} - c_{jk}[(I+c)^{-1}\, h]_i\,,
\end{equation}
where $h = (h_1,\dots,h_n)$ is defined by $h_i := \sum_q b_{qiq}$.
\end{enumerate}
In particular, if $c$ is an orthogonal projector of rank $d$
or $c = \int_{\G} S \, d \nu (S)$ for a probability measure $\nu$ on $\G$, then $c$ is positive semidefinite and symmetric. Moreover, there exists a dimensional constant $C_0 > 0$ such that
\begin{equation} \label{eqUniformBoundAijkBetaijk}
\| (I + c)^{-1} \| \leq C_0 \quad \text{and} \quad \sup_{ijk} |a_{ijk}| \leq C_0 \sup_{ijk} |b_{ijk}| \: .
\end{equation}
\end{lemma}

\begin{proof}
We first order the $3$-tuples $(i,j,k)$ in lexicographic order. Then, we note that, except for the $(l,m,l)$ columns/rows, every other column and row of $L$ contains exactly one coefficient (the diagonal one) which is equal to $1$, while all the other coefficients are zero. Therefore one has $\det (L) = \det (L^\prime)$, where $L^\prime$ is the following matrix:
\[ \begin{array}{c|cccc|cccc|c|cccc}
        & 111 & 121 & \ldots & 1n1 & 212 & 222 & \ldots & 2n2 & \ldots & n1n & n2n & \ldots & nnn \\ 
 \hline 
 111    & 1+c_{11} & 0      &        & 0      & c_{11} & 0        &  & 0      &  & c_{11} & 0 &  & 0 \\ 
 121    & c_{21}   & 1      &        & 0      & c_{21} & 0        &  & 0      &  & c_{21} & 0 &  & 0 \\ 
 \vdots & \vdots   & \vdots & \ddots & \vdots & \vdots & \vdots   &  & \vdots &  & \vdots & \vdots &  & \vdots \\ 
 1n1    & c_{n1}   & 0      &        & 1      & c_{n1} & 0        &  & 0      &  & c_{n1} & 0 &  & 0 \\ 
\hline
 212    & 0        & c_{12} &        & 0      & 1      & c_{12}   &  & 0      &  & 0 & c_{12} &  & 0 \\ 
 222    & 0        & c_{22} &        & 0      & 0      & 1+c_{22} &  & 0      &  & 0 & c_{22} &  & 0 \\ 
 \vdots & \vdots   & \vdots &        &\vdots  &\vdots &\vdots &\ddots& \vdots &  & \vdots & \vdots &  & \vdots  \\ 
 2n2    & 0        & c_{n2} &        & 0      & 0      & c_{n2}   &  & 1      &  & 0 & c_{n2} &  & 0 \\ 
\hline
 \vdots & \vdots   & \vdots &        & \vdots & \vdots & \vdots   &  & \vdots &  & \vdots & \vdots &  & \vdots \\ 
\hline
 n1n    & 0        & 0      &        & c_{1n} & 0      & 0        &  & c_{1n} &  & 1 & 0 &  & c_{1n} \\ 
 n2n    & 0        & 0      &        & c_{2n} & 0      & 0        &  & c_{2n} &  & 0 & 1 &  & c_{2n} \\ 
 \vdots & \vdots   & \vdots &        & \vdots & \vdots & \vdots   &  & \vdots &  & \vdots & \vdots & \ddots & \vdots \\ 
 nnn    & 0        & 0      &        & c_{nn} & 0      & 0        &  & c_{nn} &  & 0 & 0 &  & 1 +c_{nn}
 \end{array} \]
We then substract the $n$ last columns $(n1n,n2n, \ldots ,nnn)$ by block to the columns $(l1l,l2l, \ldots ,lnl)$ for $l=1 \ldots n-1$. Then, the last columns are unchanged, and for $l=1 \ldots n-1$, the columns $(l1l,l2l, \ldots ,lnl)$ are constituted of $I_n$ on lines $(l1l,l2l, \ldots ,lnl)$, $-I_n$ on last lines $(n1n,n2n, \ldots ,nnn)$ and zeros elsewhere. Therefore, when adding by block to the last lines $(n1n,n2n, \ldots ,nnn)$ the lines $(l1l,l2l, \ldots ,lnl)$ for $l = 1 \ldots n-1$, the determinant of $L^\prime$ is equal to the determinant of an upper triangular block-matrix, whose diagonal blocks are $(I_n, I_n, \ldots , I_n , I_n + c)$ and eventually
\[
\det (L) = \det (I_n + c) \: .
\]

The proof of the second assertion follows by plugging \eqref{aijkexplicit} into \eqref{eq:mainSystem2} and by using the matrix identity 
\[
(I+c)^{-1} - I + c(I+c)^{-1} = 0\,.
\]


If now $c = \int_{S \in \G} S \, d\nu(S)$, then $c$ is a symmetric matrix, and for every vector $X \in \R^n$ one has
\[
cX \cdot X = \int_{S \in \G} \underbrace{ SX \cdot X }_{\geq 0} \, d\nu(S) \geq 0 \: .
\]
Moreover, if $c$ is an orthogonal projector of rank $d$, then $\det (I+c) = 2^d$ (indeed, $c$ is diagonalisable with $d$ eigenvalues equal to $1$ and the remaining ones equal to $0$) and if $c = \int_{S \in \G} S \, d\nu(S)$,
then by $\log$--concavity of the determinant on the convex set of positive definite matrices, and by Jensen inequality, one obtains
\begin{align*}
\log (2^d) &= \int_{\G} \log \left( \det (I + S) \right) \, d \nu(S) \leq \log \left( \det \Big(\int_{\G} (I + S) \, d \nu(S)\Big) \right) = \log \left( \det(I + c) \right) \: , 
\end{align*}
and thus $\det (I + c) \geq 2^d$.
%
Being $S$ a projector, we have $| S_{jk} | \leq 1$ and thus
\[
| c_{jk} | = \left| \int_{\G} S_{jk} \, d \nu_x (S) \right| \leq \int_{\G} | S_{jk} | \, d \nu_x (S) \leq 1 \: ,
\]
therefore we obtain $\| I +c \|_{\infty} \leq 2$. Moreover $\det (I+c) \geq 2^d$, hence it is easy to bound $\| (I+c)^{-1} \|$
using the formula
\[
(I + c)^{-1} = \frac{1}{\det (I+c)} \mathrm{comatrix} (I+c)^T \: .
\]
This concludes the proof.
%
%
\end{proof}


We fix the notations we will stick to thereafter. For a $d$--varifold $V=\V \otimes \nu_x$ with bounded variations $\delta_{ijk} V$, we let $\beta^V = \left( \beta_{ijk}^V \right)_{ijk}$ be such that $\delta_{ijk} V = - \beta_{ijk}^V \, \V + \delta_{ijk} V_s$. Then we set $c^V = \int_{S \in \G} S \, d \nu_x(S)$, which is defined for $\V$--almost every $x$. 
Applying Lemma~\ref{lemma:systemeInversible2} to the linear system
\begin{equation} \label{eqExactSystem}
a_{ijk} + c_{jk}^V(x) \sum_{l} a_{lil} = \beta_{ijk}^V(x) \: , \quad  \text{for } i,j,k=1 \ldots n
\end{equation}
we find that it admits a unique solution for $\V$-almost every $x$. This motivates the following definition.

%

\begin{dfn}[Weak second fundamental form] \label{dfn:weakFundForm}
Let $\Omega \subset \R^n$ be an open set and $V$ be a $d$--varifold in $\Omega$ with locally bounded variations. Given the tensor $\{\beta_{ijk}\}$ from Definition~\ref{dfnBoundedVariations} we call \emph{weak second fundamental form} the unique solution $A^V=\{A_{ijk}^V(\cdot)\}$ of \eqref{eqExactSystem}. In particular, for all $i,j,k$ $A_{ijk}^V(\cdot)\in L^1_{loc}(\|V\|)$ and
\[
\delta_{ijk}V = -\left(A_{ijk}^V(x)+ \int_{\G} S_{jk} \, d \nu_x (S)\sum_q A_{qiq}^V(x)\right)\V+ (\delta_{ijk}V)_{s}.
\]
In view of the classical notion of second fundamental form, we also define the tensor $B^V = \{B_{ij}^{V,k}\}$ as
\begin{equation} \label{eq:explicitExpressionB}
B_{ij}^{V,k} = \frac{1}{2} \left( A_{ijk}^V + A_{jik}^V - A_{kij}^V \right) \: .
\end{equation}
\end{dfn}


Calling $A^V$ a second fundamental form (instead of $B^V$) is a certain abuse of terminology. It is actually very easy to switch between both tensors as the following result shows.

\begin{prop} \label{propWellDefined}
The weak second fundamental form $\{ B_{ij}^{V,k} \}$ defined in \eqref{eq:explicitExpressionB} is equivalently characterized by
\begin{equation} \label{eq:systemDfnB}
A^V_{ijk} = B_{ij}^{V,k} + B_{ik}^{V,j}  \quad \text{and} \quad B_{ij}^{V,k} = B_{ji}^{V,k} \qquad \text{for } i,j,k = 1 \ldots n\,.
\end{equation}
\end{prop}
\begin{proof}
Let us simplify the notation by setting $A_{ijk}= A_{ijk}^V$ and $B_{ij}^k = B_{ij}^{V,k}$. Assume first that $B_{ij}^k$ satisfies \eqref{eq:systemDfnB}, so that 
\begin{align*}
A_{ijk} + A_{jik} = B_{ij}^k + B_{ik}^j + B_{ji}^k + B_{jk}^i = 2 B_{ij}^k + B_{ki}^j + B_{kj}^i = 2 B_{ij}^k + A_{kij}
\end{align*}
and this proves \eqref{eq:explicitExpressionB}.
Conversely, if $\{B_{ij}^k\}$ satisfies \eqref{eq:explicitExpressionB}, then by the symmetry property $A_{ijk} = A_{ikj}$ (following from $c_{jk} = c_{kj}$, $\delta_{ijk} V = \delta_{ikj} V$ and thus $\beta_{ijk}^V = \beta_{ikj}^V$ and the formula \eqref{aijkexplicit}) we obtain $B_{ij}^k = B_{ji}^k$; moreover we have
\begin{align*}
B_{ij}^k + B_{ik}^j & = \frac{1}{2} \left( A_{ijk} + A_{jik} - A_{kij} \right) + \frac{1}{2} \left( A_{ikj} + A_{kij} - A_{jik}  \right) = \frac{1}{2} \left( A_{ijk} + A_{ikj} \right) = A_{ijk} \: ,
\end{align*}
so that $B_{ij}^k$ satisfies \eqref{eq:systemDfnB}.
\end{proof}
\medskip

In the case of an integral varifold, Brakke has shown that the mean curvature vector is orthogonal to the approximate tangent plane almost everywhere (see \cite{brakke}) and therefore, being $c^V$ the orthogonal projector on the approximate tangent plane in this case, one has $(I + c^V)^{-1} = I - \frac{c^V}{2}$ and therefore 
\begin{equation} \label{eqStructuralProperties}
\sum_q A_{q \cdot q}^V =  H^V \quad \text{and} \quad \sum_q A_{\cdot qq}^V = 0 \, .
\end{equation}
With this observation in force, we now introduce a modified version of weak second fundamental form, that is denoted by WSFF$^\perp$ and coincides with the WSFF whenever the varifold is integral and has bounded variations (see Proposition~\ref{propBetaPerp}). We will later show in Proposition~\ref{propStructuralPropertiesPerp} that this orthogonal weak second fundamental form WSFF$^\perp$ satisfies in addition the structural properties \eqref{eqStructuralProperties}.

\begin{dfn}[WSFF$^\perp$, Orthogonal weak second fundamental form] \label{dfnBetaAijkPerp}
	Let $\Omega \subset \R^n$ be an open set and let $V = \V \otimes \nu_x$ be a $d$--varifold in $\Omega$. For $\V$--almost every $x$, we define
	\begin{equation}\label{eq:betaperp1}
	\beta_{ijk}^{V,\perp} (x) = \beta_{ijk}^V (x) - c_{jk}^V (x) \left( c^V(x) \sum_q \beta_{q \cdot q}^V (x) \right)_i
	\end{equation}
	and $A_{ijk}^{V, \perp} (x)$ as the solution to the linear system \eqref{eq:mainSystem2}, with $b_{ijk} = \beta_{ijk}^{V, \perp} (x)$ and $c_{jk} = c_{jk}^V(x)$. The tensor $A_{ijk}^{V, \perp} (x)$ will be then referred to as the orthogonal weak second fundamental form WSFF$^\perp$ of $V$.
\end{dfn}

\begin{prop} \label{propBetaPerp}
	Let $\Omega \subset \R^n$ be an open set and let $V = v(M,\theta)$ be an integral $d$--varifold in $\Omega$ with bounded variations. Then, for $\cH^d$--almost every $x$, one has $\beta_{ijk}^{V,\perp} (x) = \beta_{ijk}^V(x)$, hence the WSFF and WSFF$^\perp$ coincide.
\end{prop}

\begin{proof}
	We know from Proposition~\ref{propStructuralProperties} that $\sum_q \beta_{q \cdot q}^V = H^V$ and by Brakke's result (\cite{brakke}) we know that the mean curvature vector of an integral varifold is orthogonal to the approximate tangent plane almost everywhere. Therefore, for $\cH^d$--almost every $x$,
	\[
	c^V(x) \sum_q \beta_{q \cdot q}^V (x) = 0 \: ,
	\]
	which together with \eqref{eq:betaperp1} gives the conclusion.
\end{proof}
\medskip

\section{A comparison with Hutchinson's Generalized Second Fundamental Form}
\label{sectionHutchinsonSecondFundamentalForm}
In the previous section we have introduced the notion of weak second fundamental form (WSFF). This notion corresponds to a modification of the \textit{generalized second fundamental form} that was first proposed and studied by Hutchinson in \cite{Hutchinson,Hutchinson2}. As we will explain later on, our WSFF corrects some critical aspects that are present in Hutchinson's definition, it is more consistent with Allard's definitions of first variation and of generalized mean curvature, and it is well-suited for regularization and approximation, according to the scheme proposed in \cite{BuetLeonardiMasnou}.  

In order to better motivate our WSFF, a comparison with Hutchinson's generalized tensor is necessary. We start recalling that Hutchinson's tensor arises from the application of the tangential divergence theorem \eqref{eq:divthmonmanifold} ($M$ being a smooth $d$-submanifold of $\R^n$) to vector fields of the form $X_i(x,S) = \phi (x,S) e_i$, with $\phi \in \xC_c^1 (\Omega \times {\rm M}_n (\R) )$, ${\rm M}_n (\R)$ being the space of real matrices of size $n \times n$ and $e_i$ being the $i$-th vector of the canonical basis of $\R^n$. Indeed, one obtains in this case an identity similar to \eqref{eqHutchinsonLinearP_2}, that is,
\begin{equation} \label{eqSmoothDivThmHutchinson}
- \int_M \nabla^{M} \phi (y,P(y)) \cdot e_i \: d \cH^d(y) = \int_M \Bigg( \sum_{j,k = 1}^n  A_{ijk} (y) D_{jk}^\ast \phi (y,P(y))  + \phi(y,P(y)) \sum_{q = 1}^n A_{qiq}(y) \Bigg) d \cH^d(y)
\end{equation}
where $D_{jk}^\ast \phi$ is the partial derivative of $\phi$ with respect to the variable $S_{jk}$, and $\nabla^M \phi$ is the tangential gradient of $\phi(y,S)$ with respect to $y$. Here, $M$ is assumed to be a $d$-submanifold of $\R^n$ without boundary, and the projection onto its tangent plane at $x$ is denoted by $P(x) = \left( P_{jk} (x) \right)_{jk}$.

Extending this identity to the varifold setting led Hutchinson to Definition~\ref{dfnHutchinsonWeakSecondFundamentalForm} below. Let us point out that \eqref{eqSmoothDivThmHutchinson} holds true for submanifolds without boundary, therefore Definition~\ref{dfnHutchinsonWeakSecondFundamentalForm} does not take into account boundary terms. However, we recall that a notion of curvature varifold with boundary has been studied in \cite{Mantegazza}.

\begin{dfn}[Hutchinson's Generalized Second Fundamental Form] \label{dfnHutchinsonWeakSecondFundamentalForm}
	Let $\Omega \subset \R^n$ be an open set and let $V$ be an integral $d$--varifold in $\Omega$. We say that $V$ admits a generalized second fundamental form in the sense of Hutchinson if there exists a family $\left( A_{ijk} \right)_{ijk=1 \ldots n} \in \xL^1(V)$ such that for all $\phi \in \xC_c^1 (\Omega \times {\rm M}_n (\R))$,
	\begin{equation} \label{eqHutchinsonBPI}
	\int_{\Omega \times \G} \Bigg( \nabla^S \phi (y,S) \cdot e_i + \sum_{j,k = 1}^n D_{jk}^\ast \phi (y,S) A_{ijk}(y,S) + \phi(y,S) \sum_{q = 1}^n A_{qiq} (y,S) \Bigg) d V(y,S) = 0 \: .
	\end{equation}
\end{dfn}
A first observation concerning Hutchinson's tensor is that, in general, it may depend not only on the spatial variable $x$, but also on the Grassmannian variable $S$. On the one hand, this seems not fully consistent with the fact that Allard's generalized mean curvature only depends on $x$. On the other hand, Hutchinson himself implicitly shows in \cite[Proposition 5.2.2]{Hutchinson} that his tensor ultimately depends upon $x$ - and not on $S$ - as soon as the varifold is rectifiable. 

A second observation is about the existence of the tensor as a distribution, which is a delicate issue even if one restricts the analysis to integral varifolds. Here one finds a structural obstruction. Indeed, due to the form of the middle term in the left-hand side of \eqref{eqHutchinsonBPI}, where the $jk$-derivative of $\phi$ multiplies $A_{ijk}$, it is not possible to interpret the identity as a distributional definition of the tensor $A$. However, we note that the class of test functions ensuring \textit{uniqueness} of the tensor $A$ is strictly smaller than the class used in the definition of this tensor, see \cite[Proposition 5.2.2]{Hutchinson}. To be more specific, that definition uses test functions which are possibly nonlinear in the Grassmannian variable, however linearity in the Grassmannian variable is enough to prove uniqueness. From this observation we infer that Hutchinson's definition contains a \textit{core distributional notion} (the one leading to our WSFF), plus extra constraints coming from test functions $\phi(x,S)$ that are nonlinear in $S$, which make Hutchinson's definition more rigid. As we have said in the introduction, this rigidity affects the structure of the singularities of integral varifolds whose curvature tensor $A$ is bounded. Indeed, a non-quantitative version of Lemma 3.3 coupled with Corollary 3.5 in \cite{Hutchinson2} shows that an integral varifold $V$, with curvature tensor in $\xL^p_{loc}$ for $p>d$, only admits blow-ups in the form of finite unions of $d$-planes with multiplicities.


We now restrict our analysis to rectifiable varifolds, as actually done in \cite{Hutchinson}. We first prove the following proposition.
\begin{prop}[WSFF versus Hutchinson's tensor: the case of rectifiable varifolds]\label{prop:CoincideOnRectifiable}
If a rectifiable varifold $V$ in $\R^n$ admits a generalized second fundamental form in the sense of Hutchinson (Definition~\ref{dfnHutchinsonWeakSecondFundamentalForm}), then it also admits a weak second fundamental form in the sense of Definition~\ref{dfn:weakFundForm}, and the two tensors coincide $\V$-almost everywhere.
\end{prop}
\begin{proof}
Let $V = v(M,\theta)$ be a $d$--rectifiable varifold in $\R^n$ and let us denote, for $\V$--almost every $x$, by $P(x)$ its approximate tangent plane at $x$. We assume that there exists $\tilde{A}_{ijk} \in \xL^1 (\R^n \times \G, V)$ sastisfying \eqref{eqHutchinsonBPI}. Taking test functions of the form $\psi_{jk} (x,S) = \phi(x) S_{jk}$, one has 
\[
D_{j^\prime k^\prime}^\ast \psi_{jk} (x, S) = \left\lbrace \begin{array}{ll}
1 & \text{if } (j^\prime, k^\prime) = (j,k) \\
0 & \text{otherwise}
\end{array} \right.
\]
By \eqref{eqHutchinsonBPI} we obtain that for all $i,j,k = 1, \ldots n$ and $\phi \in \xC_c^1(\R^n)$,
	\begin{align*}
	- \int_{\R^n \times \G} \hspace{-10pt} S_{jk} \nabla^S \phi(x) \cdot e_i \: dV(x,S) & = \int_{\R^n \times \G} \hspace{-5pt} \Bigg( \tilde{A}_{ijk} (x,S)  + S_{jk} \sum_{q=1}^n \tilde{A}_{qiq} (x,S)  \Bigg) \phi(x) \, dV(x,S) \\
	& = \int_{\R^n} \Bigg( \tilde{A}_{ijk} (x, P(x))  + P_{jk}(x) \sum_{q=1}^n \tilde{A}_{qiq} (x, P(x))  \Bigg) \phi(x) \, d \V(x) .
	\end{align*}
Therefore $V$ has bounded variations $\delta_{ijk} V$ (see Definition~\ref{dfnVariations}) with no singular part:
\begin{equation} \label{eqHutchinsonSyst}
\delta_{ijk} V = - \beta_{ijk} \V \quad \text{with} \quad \beta_{ijk}(x) = \tilde{A}_{ijk} (x, P(x))  + P_{jk}(x) \sum_{q=1}^n \tilde{A}_{qiq} (x, P(x)) \text{ for } \V\text{--a.e. } x
\end{equation}
Since $V$ has bounded variations, it has also a WSFF (in the sense of Definition~\ref{dfn:weakFundForm}) that we denote by $A_{ijk}^V$. Then we infer from \eqref{eqHutchinsonSyst} that for $\V$--a.e. $x$, $\tilde{A}_{ijk} (x, P(x))$ and $A_{ijk}^V (x)$ are both solutions of the linear system \eqref{eqExactSystem}, therefore they must coincide by uniqueness.
\end{proof}
\begin{remk}
In view of~\eqref{eqSmoothDivThmHutchinson} and Hutchinson's definition, an obvious consequence of the above proposition is that our weak second fundamental form coincides with the classical second fundamental form whenever $V=v(M,c)$ with $M$ a smooth $d$--submanifold of $\R^n$ and $c>0$ a constant.
\end{remk}

\subsection*{Structural differences between the WSFF and Hutchinson's generalized second fundamental form: a $1$-dimensional example in the plane}
We are now going to see that the essential difference between our WSFF and Hutchinson's tensor shows up in the structure of admissible singularities of a curvature varifold. We remark that in \cite{Hutchinson2} the following fact is shown: if an integral varifold $V$ admits a generalized second fundamental form in $\xL^p$ for $p>d$, then any tangent cone at every point of the support of $\V$ consists of a finite union of $d$--planes with constant integral multiplicities. This follows from a key monotonicity identity proved in \cite[Section 3.4]{Hutchinson2}. In particular, a junction of half-lines, seen as a $1$--varifold with multiplicity $1$, is a curvature varifold in Hutchinson's sense if and only if it is a finite union of full lines. In the example below, we compute our WSFF and Hutchinson's tensor in the case of 9 half-lines issuing from the origin through the vertices of a regular $9$-gon, obtaining that the WSFF is identically zero, while Hutchinson's tensor is not globally defined.  

Let us consider the rectifiable $1$--varifold $V$ associated with a union of $N$ half lines ${(D_l)}_{l=1}^N$ joining at $0$, of constant multiplicity $1$, and directed by a unit vector ${( u_l ) }_{l = 1}^N \subset \R^2$. On each half line $D_l$, the projection matrix $P^l = (P_{jk}^l)_{jk}$ onto the tangent line satisfies
$P_{jk}^l = (u_l \cdot e_j)(u_l \cdot e_k)$ for $j,k = 1,2$.
Let us compute the $G$-linear variation $\delta_{ijk} V$, for $\phi \in \xC_c^1 (\R^2)$ and $i,j,k = 1,2$:
\begin{align*}
\delta_{ijk}V (\phi) & = \sum_{l = 1}^N \int_{D_l} P_{jk}^l (\nabla \phi \cdot u_l) (u_l \cdot e_i)\, d \cH^1 = \sum_{l = 1}^N P_{jk}^l (u_l \cdot e_i)  \int_{D_l} (\nabla \phi \cdot u_l) \, d \cH^1   \\
 & =  - \sum_{l = 1}^N P_{jk}^l (u_l \cdot e_i) \phi(0) = - \sum_{l = 1}^N (u_l \cdot e_i)(u_l \cdot e_j)(u_l \cdot e_k) \delta_0 (\phi).
\end{align*}
Therefore $V$ has bounded variations and
\begin{equation} \label{eqVariationsJunction}
\delta_{ijk} V = - \sum_{l = 1}^N (u_l \cdot e_i)(u_l \cdot e_j)(u_l \cdot e_k) \delta_0
\end{equation}
is singular (or zero) with respect to $\V$. In particular, contrarily to Hutchinson's case, any such junctions admit a WSFF. We then wonder if some of them have zero WSFF without being a union of lines. Hence, we study the case where the WSFF of the junction is identically zero. By \eqref{eqVariationsJunction}, this property is equivalent to the following $4$ equations:
\[
\sum_{l = 1}^N (u_l \cdot e_1)^3 = 0, \quad \sum_{l = 1}^N (u_l \cdot e_2)^3 =0, \quad \sum_{l = 1}^N (u_l \cdot e_1)^2 (u_l \cdot e_2) = 0, \quad \sum_{l = 1}^N (u_l \cdot e_1)(u_l \cdot e_2)^2 = 0.
\] 
As $u_l$ are unit vectors, it is also equivalent to
\[
\sum_{l = 1}^N (u_l \cdot e_1)^3 = 0, \quad \sum_{l = 1}^N (u_l \cdot e_2)^3 =0, \quad \sum_{l = 1}^N (u_l \cdot e_1) = 0, \quad \sum_{l = 1}^N (u_l \cdot e_2) = 0.
\]
If $u_l = (\cos \alpha_l , \sin \alpha_l)$ for $\alpha_l \in ] - \pi , \pi ]$, we end up with the conditions
\[
\sum_{l = 1}^N \exp (i \alpha_l) = 0 \quad \text{and} \quad \sum_{l = 1}^N \exp (3 i \alpha_l) = 0 \: .
\]
By taking $N=9$ and $\alpha_l = \frac{2 l \pi}{N}$, we obtain a regular junction of $9$ half--lines (which is not a union of full lines), whose weak second fundamental form is identically zero. At the same time, Hutchinson's tensor cannot be globally defined, otherwise it should also be identically zero by Proposition~\ref{prop:CoincideOnRectifiable}, and thus we would reach a contradiction with the admissible singularities of a varifold with zero curvature (see \cite[Corollary 3.5]{Hutchinson2} and the discussion following Definition~\ref{dfnHutchinsonWeakSecondFundamentalForm}).


\begin{remk}[Interpretation in terms of weak differentiability of the approximate tangent plane]
Theorem $15.6$ of \cite{menne} states that an integral varifold is a curvature varifold in Hutchinson's sense if and only if $\delta V$ is a Radon measure absolutely continuous with respect to $\|V\|$ and the approximate tangent plane $P(x)$ is a weakly $V$-differentiable function in the sense of \cite[Definition 8.3]{menne}. 
Thanks to this characterization, in \cite{MenneSharrer2018} the notion of curvature varifold is extended to diffuse-type varifolds.
\end{remk}

\section{Properties of the weak second fundamental form}

\subsection{Structural properties}
\label{sectionStructuralProperties}
In the smooth case, the functions $\beta_{ijk}^V$ and $A_{ijk}^V$ satisfy some structural properties. We check in the next proposition which ones are still valid in our extended setting.

\begin{prop} \label{propStructuralProperties}
Let $\Omega \subset \R^n$ be an open set and $V$ be a $d$--varifold in $\Omega$ with bounded variations. Then, for all $i, j, k = 1, \ldots, n$,
\begin{enumerate}
\item $\displaystyle \delta_{ijk} V = \delta_{ikj} V$, $\beta_{ijk}^V = \beta_{ikj}^V$ and $A_{ijk}^V = A_{ikj}^V$,
\item  $\displaystyle \sum_q \beta_{iqq} = d H_i^V$ and $\displaystyle \sum_q \beta_{qiq}^V = H_i^V$ where $H^V = (H_1^V, \ldots, H_n^V) = - \frac{\delta V}{\|V\|}$ is the generalized mean curvature vector,
\item $\displaystyle \sum_q A_{q \cdot q}^V = (I + c^V)^{-1} H^V$ and $\displaystyle \sum_q A_{\cdot qq}^V = d c^V (I+c^V)^{-1} H^V$.
\end{enumerate}
\end{prop} 

\begin{proof}
The first assertion follows from $P_{jk} = P_{kj}$ for every $P \in \G$. In order to prove the second assertion, for $\phi \in \xC_c^1 (\Omega)$ we compute 
\[
\sum_q \delta_{\cdot qq} V (\phi) = \int_{\Omega \times \G} \underbrace{ \sum_q S_{qq} }_{= d} \nabla^S \phi(y) \, dV(y,S) = d \delta V (\phi)
\]
and
\[
\sum_q \delta_{q \cdot q} V(\phi) = \int_{\Omega \times \G} \underbrace{ \sum_q S_{\cdot q} \left( \nabla^S \phi(y) \right)_q }_{ \displaystyle = S \nabla^S \phi(y)} \, dV(y,S) = \int_{\Omega \times \G} \nabla^S \phi(y) \, dV(y,S) = \delta V (\phi) \: . 
\]
Consequently, we deduce that $\sum_q \delta_{\cdot qq} V = d \delta V$ and  $\sum_q \delta_{q \cdot q} V = \delta V$. By taking the Radon-Nikodym derivative with respect to $\V$, the analogous equalities hold for $\beta^V$ and $H^V$. The last assertions are then direct consequences of formula \eqref{aijkexplicit}:
\begin{align*}
\sum_q A_{q \cdot q}^V = \sum_q \left( \beta_{q \cdot q}^V - c_{\cdot q} ( (I+c)^{-1} H^V )_ q \right) = H^V - c(I+c)^{-1} H^V
= (I+c)^{-1} H^V
\end{align*}
and
\begin{align*}
\sum_q A_{\cdot qq}^V = \sum_q \beta_{\cdot qq}^V - \sum_q c_{qq} (I+c)^{-1} H^V = d H^V - d (I+c)^{-1} H^V = d c (I+c)^{-1} H^V\,.
\end{align*}
\end{proof}
We recall that the following identities were proved in Proposition~\ref{propStructuralProperties}:
\[
\sum_q A_{q \cdot q}^V = (I+c^V)^{-1} H^V \quad \text{and} \quad \sum_q A_{\cdot qq}^V = d c^V (I+c^V)^{-1} H^V \:.
\]
Therefore, the identities 
\begin{equation} \label{eqStructuralAijk2}
 \sum_q A_{q \cdot q}^V = H^V \quad \text{and} \quad \sum_q A_{\cdot qq}^V = 0
\end{equation}
are true for an integral varifold, but not for a general varifold. Owing to
$(I+c^V)^{-1} = I - (I+c^V)^{-1} c^V$, we get that \eqref{eqStructuralAijk2} hold if and only if $c^V H^V = 0$. Nevertheless, in the following proposition we show that the identities \eqref{eqStructuralAijk2} are satisfied by $A_{ijk}^{V, \perp}$ (see Definition \ref{dfnBetaAijkPerp}).

\begin{prop} \label{propStructuralPropertiesPerp}
Let $\Omega \subset \R^n$ be an open set and $V$ be a $d$--varifold in $\Omega$ with bounded variations. Then,
\[
\sum_q A_{q \cdot q}^{V, \perp } = H^{V, \perp} \quad \text{and} \quad  \sum_q A_{\cdot qq}^{V, \perp} = 0\,,
\]
where we have set $H^{V, \perp} = (I - c^V) H^V$.
\end{prop}

\begin{proof}
Indeed, dropping the $V$--exponent for more simplicity, and recalling that $\sum_q \beta_{q \cdot q} = H$, we have
\begin{align*}
\sum_q \beta_{qiq}^\perp = \sum_q \beta_{qiq} - \sum_q c_{iq} \left( c \sum_l \beta_{l \cdot l} \right)_q 
= \sum_q \beta_{qiq} - c \left( c \sum_l \beta_{l \cdot l} \right)_i  = ((I - c^2)H )_i
\end{align*}
and then, 
\begin{align}
A_{ijk}^{\perp} & = \beta_{ijk}^\perp - c_{jk} \left( (I+c)^{-1} \sum_q \beta_{q\cdot q}^\perp \right)_i \nonumber \\
& = \beta_{ijk} - c_{jk} \left( c \sum_q \beta_{q \cdot q} \right)_i - c_{jk} \left( (I+c)^{-1} (I-c^2)H \right)_i \nonumber \\
& = \beta_{ijk} - c_{jk} \left( c H + (I-c) H \right)_i \nonumber \\
& = \beta_{ijk} - c_{jk} H_i \: , \label{eqExplicitAijkPerp}
\end{align}
We can now compute
\begin{align*}
\sum_q A_{qiq}^\perp = \sum_q \beta_{q i q} - \sum_q c_{iq} H_q  = H_i - (cH)_i = H_i^\perp
\end{align*}
and
\begin{align*}
\sum_q A_{iqq}^\perp = \sum_q \beta_{iqq} - \sum_q c_{qq} H_i = d H_i - d H_i =0\,,
\end{align*}
as wanted.
\end{proof}

\subsection{Convergence and compactness}
\label{sectionCvCompactness}

We start with a direct consequence of Allard's compactness theorem, see also \cite{BuetLeonardiMasnou}.
\begin{prop}
Let $\Omega \subset \R^n$ be an open set and let $(V_h)_h$ be a sequence of $d$--varifolds in $\Omega$, with bounded variations. Assume that for $i,j,k = 1 \ldots n$,
\begin{equation} \label{eqAllardTypeAssumption}
\sup_h \{ \| V_h \| (\Omega) + | \delta_{ijk} V_h |(\Omega) \} < + \infty \: ,
\end{equation}
then, there exists a subsequence $(V_{h_l})_l$ weakly--$\ast$ converging to a $d$--varifold $V$ with bounded variations satisfying
\[
| \delta_{ijk} V| (\Omega ) \leq \liminf_l | \delta_{ijk} V_{h_l}| (\Omega ) \: .
\]
Moreover, if the $V_h$ are integral varifolds, then $V$ is integral.
\end{prop}

\begin{proof}
Let $(e_1, \ldots , e_n)$ be the canonical basis of $\R^n$, then
\[
\delta V = \sum_{i=1}^n \sum_{j=1}^n \delta_{ijj} V \, e_i \: .
\]
Therefore \eqref{eqAllardTypeAssumption} implies
\[
\sup_h \{ \| V_h \| (\Omega) + | \delta V_h |(\Omega) \} < + \infty \: ,
\]
and it only remains to apply Allard's compactness theorem (\cite{Allard72}) and use the lower semicontinuity of $|\delta_{ijk} V |(\Omega)$.
\end{proof}

Next, we deal with some convergence properties of the WSFF.
Let $\Omega \subset \R^n$ be an open set and let $(V_h)$ be a sequence of $d$--varifolds weakly--$\ast$ converging to a $d$--varifold $V$. Assuming that the variations of $V_h$ are bounded and that $\sup_h \left| \delta_{ijk} V_h \right| (\Omega)  < +\infty$ is enough to conclude that $V$ has bounded variations. Thus $V_h$ (resp. $V$) admits a WSFF (in the sense of Definition~\ref{dfn:weakFundForm}), denoted by $A_{ijk}^h$ (resp. $A_{ijk}$). The following theorem focuses on the convergence of the measures $A_{ijk}^h V_h$, as $h\to\infty$.

\begin{prop}[Convergence of the WSFF]
Let $\Omega \subset \R^n$ be an open set and let $(V_h)_h$ be a sequence of $d$--varifolds weakly--$\ast$ converging to a rectifiable $d$--varifold $V$. Assume that there exists constants $C_0>0$  and $p>1$, such that $\delta_{ijk} V_h << \| V_h \|$ and
\begin{equation} \label{eqConvergence0}
\int_{\Omega} |A_{ijk}^h(x)|^p \, d \| V_h \| (x) \le C_0
\end{equation}
for all $h\in \N$. 
Then $\delta_{ijk}V$ is absolutely continuous with respect to $\V$, and the measures
$A_{ijk}^h V_h$ weakly--$\ast$ converge to the measure $A_{ijk} V$ as $h\to\infty$.
\end{prop}

\begin{proof}
Let $M$ be a $d$--rectifiable set endowed with a multiplicity function $\theta$, such that $V=v(M,\theta)$. 
Thanks to \eqref{eq:mainSystem2} we have
\[
\sup_{i,j,k} |\beta_{ijk}^h(x)| \le C_n \sup_{i,j,k} |A_{ijk}^h(x)| 
\] 
for $\| V_h \|$-almost all $x\in \Omega$ and $h\in \N$, 
with $C_n$ only depending on the dimension $n$. Therefore, by \eqref{eqConvergence0} and H\"older's inequality we infer that 
\begin{equation} \label{eqConvergence1}
\left| \delta V_{ijk}^h \right| (K) = \left|\int_K \beta_{ijk}^h\, d\V\right| \leq \V(K)^\frac{1}{p'}  \left(\int_{\Omega} |\beta_{ijk}^h|^p \, d \| V_h \| \right)^\frac{1}{p} \le C_0 C_n \V(K)^\frac{1}{p'}
\end{equation}
for every compact set $K\subset \Omega$. 
Combining the fact that $\delta_{ijk} V_h (\phi)$ converges to $\delta_{ijk} V (\phi)$ for all $\phi \in \xC_c^1 (\Omega)$ with the above uniform control \eqref{eqConvergence1}, we obtain that $\delta_{ijk} V_h$ weakly--$\ast$ converges to $\delta_{ijk} V$. We can now apply Example $2.36$ in \cite{ambrosio} to conclude that $\delta_{ijk} V << \V$ and moreover
\[
\int_{\Omega} |\beta_{ijk}(x)|^p \, d \| V \| (x) \leq \liminf_h \int_{\Omega} |\beta_{ijk}^h(x)|^p \, d \| V_h \| (x) .
\]
On the other hand, from \eqref{eqConvergence0} we similarly deduce that $\displaystyle \sup_h \int_K |A_{ijk}^h| \, d \|V_h \| < +\infty$ for every fixed compact $K\subset \Omega$. Therefore there exists a subsequence (which we do not relabel) that weakly--$\ast$ converges to a limit Radon measure which is absolutely continuous with respect to $\V$. In other words, there exists $f_{ijk} \in  \xL^p (\V)$ such that $A_{ijk}^h \| V_h \|$ weakly--$\ast$ converges to $f_{ijk} \V$. Note that
\[
\sup_h \int_{\Omega \times \G} |A_{ijk}^h(x)|^p \, d  V_h  (x, P) = \sup_h \int_{\Omega} |A_{ijk}^h(x)|^p \, d \| V_h \| (x) < +\infty ,
\]
and thus (up to extraction) there exists $g_{ijk} \in \xL^p (V)$ such that $A_{ijk}^h V_h$ weakly--$\ast$ converges to $g_{ijk} V$. 

We now show that $f_{ijk}(x)=g_{ijk}(x,P(x))$ for $\V$-almost every $x$. First, for every $\phi \in \xC_c^0(\Omega)$, by taking the limit as $h \to \infty$ in the equality
\[
\int_\Omega A_{ijk}^h(x) \phi(x) \, d \| V_h \| (x) = \int_{\Omega \times \G} A_{ijk}^h(x) \phi(x) \, d  V_h (x,S)\,,
\]
and by representing the limit measure $V$ via the Young measures $\{\nu_x\}$, we get
\[
\int_\Omega f_{ijk}(x) \phi(x) \, d \| V \| (x) \hspace{-1pt} = \hspace{-1pt} \int_{\Omega \times \G} g_{ijk}(x,S) \phi(x) \, d  V (x,S) = \int_{\Omega } \int_{\G} g_{ijk}(x,S) \, d \nu_x (S) \phi(x) \, d  \V (x)
\]
so that for $\V$--almost every $x$, $\displaystyle f_{ijk} (x) = \int_{S \in \G} g_{ijk} (x,S) \, d \nu_x(S) = g_{ijk} (x,P(x))$, as wanted.

By definition of $A_{ijk}^h$, for every $\phi \in \xC_c^0 (\Omega)$ and for all $h$ we have
\begin{align*}
\int_\Omega \left( A_{ijk}^h (x) \right. & \left. + \int_{S \in \G} S_{jk} \, d \nu_x^h(S) \sum_{q} A_{qiq} (x) \right) \phi (x) \, d \|V_h\|(x) \\ & = \int_\Omega  A_{ijk}^h (x) \phi (x) \, d V_h(x,S)  +  \sum_{q} \int_{(x,S) \in \Omega \times \G}  A_{qiq}^h (x)  \phi (x) S_{jk} \, d V_h (x,S)\\
& = \int_\Omega \beta_{ijk}^h(x) \phi(x) \, d \|V_h\|(x).
\end{align*}
Taking again the limit as $h \to \infty$, we get
\[
\int_{\Omega }  f_{ijk}(x) \phi (x) \, d \|V\|(x)  +  \sum_{q} \int_{\Omega \times \G}  g_{qiq} (x,S)  \phi (x) S_{jk} \, d V (x,S)
 = \int_\Omega \beta_{ijk}(x) \phi(x) \, d \|V\|(x),
\]
and consequently, for $\V$--almost every $(x)$,
\begin{align*}
f_{ijk}(x) & + \int_{S \in \G} S_{jk} \sum_{q} g_{qiq} (x,S)  \, d \nu_x (S) = f_{ijk} (x) + P(x) \sum_q g_{qiq} (x,P(x)) \\
& = f_{ijk} (x) + P(x) \sum_q f_{qiq} (x) = \beta_{ijk}(x).
\end{align*}
By uniqueness of the solution to the linear system \eqref{eq:mainSystem2} (see Lemma~\ref{lemma:systemeInversible2}) we obtain 
\[
g_{ijk} (x,P(x)) = f_{ijk} (x) = A_{ijk} (x)
\] 
for $\V$--almost every $x$, hence $g_{ijk}(x,S) = A_{ijk}(x)$ for $V$--almost every $(x,S)$, which concludes the proof.
\end{proof}

\section{Approximate weak second fundamental form}
\label{Section:RegWSFF}
Let us fix some notations that we will use all along the paper. 
We fix three non-negative functions $\rho, \xi,\eta \in \R_+ \rightarrow \R_+$ of class $C^1$, with the properties listed below:
\begin{itemize}
	\item for all $t \geq 1$, $\rho (t) = \xi(t) = \eta(t) = 0$; 
	\item $\rho$ is decreasing, $\rho^\prime (0) =0$, and $\int_{\R^n} \rho(|x|) \, dx = 1$;
	\item $\xi(t)>0$ for all $0<t<1$, and $\int_{\R^n} \xi(|x|) \, dx = 1$.
\end{itemize}
We define for $\e > 0$ and $x \in \R^n$
\begin{equation} \label{eq:dfnRadialKernels}
\rho_\e (x) = \frac{1}{\e^n} \rho \left( \frac{|x|}{\e} \right), \qquad 
\xi_\e (x) = \frac{1}{\e^n} \xi \left( \frac{|x|}{\e} \right),\qquad
\eta_\e (x) = \frac{1}{\e^n} \eta \left( \frac{|x|}{\e} \right)\,.
\end{equation}
We also define the constants
\begin{equation}
C_\rho = d \omega_d \int_0^1 \rho(r)r^{d-1} \, dr\,, \qquad C_\xi = d \omega_d \int_0^1 \xi(r)r^{d-1} \, dr\,.
\end{equation}
Given an open set $\Omega\subset\R^n$ and $\e>0$ we define $\Omega_\e = \{x\in \Omega:\ \dist(x,\partial \Omega)>\e\}$.

\subsection{Approximation via mollification}

We showed in \cite{BuetLeonardiMasnou} how to define the approximate mean curvature of general varifolds using a mollification of the first variation. We extend here this approach to define the approximate second fundamental form of general varifolds. We observe that the idea of mollifying the first variation to get a smoothed mean curvature was first used by Brakke in~\cite{brakke} for varifolds with bounded first variation in order to construct their evolution by mean curvature flow. 

\label{sectionRegularization}
%

\begin{dfn}[Regularized variations]\label{dfn:RegGLinVar} 
Let $\Omega \subset \R^n$ be an open set and let $V$ be a $d$--varifold in $\Omega$. Given $i,j,k \in \{1, \ldots, n\}$ and $\e>0$, for any vector field $\varphi \in \xC_c^1 (\Omega_\e)$ we set
\begin{equation}\label{def:regGlinvar}
\delta_{ijk} V \ast \rho_\epsilon (\varphi
) := \delta_{ijk} V (\varphi \ast \rho_\epsilon) = \int_{\Omega \times \G}S_{jk} \,  \nabla^S  (\varphi\ast \rho_\epsilon) (y) \cdot e_i \, dV(y,S) \:.
\end{equation}
We say that $\delta_{ijk} V \ast \rho_\epsilon$ is a \emph{regularized $G$-linear variation} of $V$. 
\end{dfn}
Of course \eqref{def:regGlinvar} defines $\delta_{ijk} V \ast \rho_\epsilon$ in the sense of distributions. The following proposition, which is a simple adaptation of Proposition 4.2 in~\cite{BuetLeonardiMasnou}, shows that $\delta_{ijk} V \ast \rho_\epsilon$ can be represented by a smooth vector field with locally bounded $L^{1}$-norm when $V$ has locally finite mass in $\Omega$.
\begin{prop}[Representation of the regularized variations]\label{expression_of_delta_ijk_V^epsilon}
Let $\Omega \subset \R^n$ be an open set, $V$ a $d$--varifold in $\Omega$, and let $i,j,k \in \{1, \ldots, n\}$. Then $\delta_{ijk} V \ast \rho_\epsilon$ is represented by the continuous vector field
\begin{equation} \label{regGlinvar-rep}
\delta_{ijk} V \ast \rho_\epsilon (x) =  \int_{\Omega \times \G} \!\!\!\!S_{jk}\nabla^S \rho_\epsilon (y-x)\cdot e_i \, dV(y,S) = \frac{1}{\epsilon^{n+1}}\int_{\Omega \times \G} S_{jk}\nabla^S \rho_{1} \left(\frac{y-x}{\epsilon}\right)\cdot e_i \, dV(y,S)
\end{equation}
defined for $x\in \Omega_\e$. In addition, $\delta_{ijk} V \ast \rho_\epsilon \in \xL^1_{loc}(\Omega_\e)$.
\end{prop}
\begin{proof}
By Fubini-Tonelli's theorem we get from~\eqref{def:regGlinvar} for every $\varphi \in \xC_c^1 (\Omega_\e)$:
\begin{align*}
 \delta_{ijk} V \ast \rho_\epsilon ( \varphi) & = \int_{\Omega \times \G} S_{jk}(\varphi \ast \nabla^S \rho_\epsilon)(y)\cdot e_i \, dV(y,S),\\
& = \int_{\Omega \times \G}S_{jk} \left(\int_{x \in \Omega_\e} \varphi(x) \nabla^S \rho_\epsilon(y-x) \, d \cL^n(x)\right)\cdot e_i \, dV(y,S), \\
& = \int_{x \in\Omega_\e} \varphi(x) \left( \int_{\Omega \times \G}  S_{jk} \nabla^S \rho_\epsilon(y-x)\cdot e_i \, dV(y,S) \right) \, d \cL^n (x) \: ,
\end{align*}
which proves \eqref{regGlinvar-rep}. The fact that $\delta_{ijk} V \ast \rho_\epsilon \in \xL^1_{loc}(\Omega_\e)$ is an immediate consequence of $\nabla\rho_{\e}$ being bounded in $\R^{n}$.
\end{proof}

\begin{remk}\rm 
We emphasize that $\delta_{ijk} V \ast \rho_\epsilon$ is in $L^{1}_{loc}(\Omega_\e)$ even when $\delta_{ijk} V$ is not locally bounded.
\end{remk}

In view of the above representation of the $G$-linear variation, given any $d$--varifold $V = \V \otimes \nu_x$ and given $\e > 0$, the following quantities are defined for $\V$--almost every $x\in\Omega_\e$:
\begin{equation}\label{eq:betaeps-ceps}
\beta_{ijk}^{V,\e} (x) = - \frac{C_\xi}{C_\rho}\frac{\delta_{ijk} V \ast \rho_\e (x)}{\V \ast \xi_\e (x)}  \quad \text{and} \quad c_{jk}^{V,\e} (x) =  \frac{ \left( \displaystyle\int_{\G} S_{jk} \, d \nu_\cdot (S) \V \right) \ast \eta_\e (x)}{\V \ast \eta_\e (x)}\,.
\end{equation}
Recall that in \cite{BuetLeonardiMasnou}, we defined a regularization of the mean curvature vector $H^V$ of a varifold $V$ with finite mass and locally bounded first variation as the approximate mean curvature vector given for $\V$--a.e. $x\in \Omega_\e$ by
\[
H^{V,\e}(x) = - \frac{C_\xi}{C_\rho} \frac{\delta V \ast \rho_\e (x)}{\V \ast \xi_\e (x)} \: .
\]
The regularizations of $\beta_{ijk}^V$ and $H^V$ are consistent in the sense that the structural identities $(2)$ in Proposition~\ref{propStructuralProperties} are preserved, i.e.
\begin{equation} \label{eqBetaHeps}
\sum_q \beta_{q i q}^{V,\e} =  H_i^{V,\e} \quad \text{and} \quad \sum_q \beta_{iqq}^{V,\e} = d H_i^{V,\e} \: .
\end{equation}

\begin{remk}
As when we defined $H^{V,\e}$ in \cite{BuetLeonardiMasnou}, we allow two different kernels for the regularization of $\delta_{ijk} V$ and $\V$ in the definition of $\beta_{ijk}^{V, \e}$. 
As evidenced for mean curvature, the same computations as done in \cite{BuetLeonardiMasnou} Section 5, suggest a specific choice of $\rho$ and $\xi$, namely the so-called natural kernel pair relation, that is
\begin{equation} \label{eq:NKP}
n \xi(s) = - s \rho^\prime(s) \, .
\end{equation}
This choice has proven to be relevant from a numerical point of view (see \cite{BuetLeonardiMasnou}, Section~8) and we shall use it in the numerical experiments performed in Section~\ref{sectionNumerics}, we recall that with this choice of kernels $\frac{C_\xi}{C_\rho} = \frac{d}{n}$.
\end{remk}
We can write $\displaystyle c^{V, \e}(x) = \int_{\Omega \times \G} S \, dW_x (y,S)$ where $W_x $ is the probability measure defined as
\[
dW_x (y,S) = \frac{\eta_\e \left( y - x \right) dV(y,S)}{\int_{\Omega \times \G} \eta_\e \left( z - x \right) d V(z,P)}\: .
\]
so that Lemma~\ref{lemma:systemeInversible2} applies to the linear system
\begin{equation} \label{eq:RegularizedSystem}
a_{ijk} + c_{jk}^{V,\e}(x) \sum_{l} a_{lil} = \beta_{ijk}^{V,\e}(x) \: , \quad  \text{for } i,j,k=1 \ldots n,
\end{equation}
which arises as the natural regularization of \eqref{eq:mainSystem2} and the unique solution is explicitly given in \eqref{aijkexplicit}. This allows to introduce a notion of approximate second fundamental form as follows:

\begin{dfn}[Approximate second fundamental form, $\e$--WSFF] \label{propDfnRegularizedSecondFundForm}
Let $\Omega \subset \R^n$ be an open set and $V = \V \otimes \nu_x$ be a $d$--varifold in $\Omega$.
For $\epsilon >0$ and for $\V$--almost every $x\in \Omega_\e$, we define $A_{ijk}^{V,\e} (x)$ as the unique solution to system \eqref{eq:RegularizedSystem} that is
\begin{equation} \label{eqExplicitAijkeps}
A_{ijk}^{V,\e}  =  \beta_{ijk}^{V,\e} - c_{jk}^{V,\e} \left( (I + c^{V,\e})^{-1} H^{V,\e} \right)_i \: .
\end{equation}
We call $\{A_{ijk}^{V,\e}\}$ the approximate second fundamental form, or the $\e$--WSFF. With a slight abuse of terminology, we will also call approximate second fundamental form the tensor $\{B_{ij}^{k,\e}\}$ defined by $\displaystyle B_{ij}^{k,\e} = \frac{1}{2} ( A_{ijk}^{V,\e} + A_{jik}^{V,\e} - A_{kij}^{V,\e} )$  (see~\eqref{eq:explicitExpressionB}).
\end{dfn}

\begin{remk}
We stress that the $\e$--WSFF can be defined for \emph{any} varifold $V$, even though $V$ does not have bounded variations. 
\end{remk}

We denote for brevity
\[
\beta^{V} = (\beta_{ijk}^{V})_{ijk}, \quad c^{V} = (c_{jk}^{V})_{jk}, \quad A^{V} = (A_{ijk}^{V})_{ijk} \quad \text{and} \quad B^{V} = (B_{ij}^{k})_{ijk}
\]
and
\[
\beta^{V,\e} = (\beta_{ijk}^{V,\e})_{ijk}, \quad c^{V,\e} = (c_{jk}^{V,\e})_{jk}, \quad  A^{V,\e} = (A_{ijk}^{V,\e})_{ijk} \quad \text{and} \quad B^{V,\e} = ( B_{ij}^{k,\e} )_{ijk} \: .
\]
When there is no ambiguity, we might drop the part of the exponent referring to the varifold $V$. 

\subsection{Consistency of the approximate second fundamental form}
\label{sectionConsistency}

We justify in Proposition~\ref{prop_pointwise_convergence} the definition of $\e$--WSFF (Definition~\ref{propDfnRegularizedSecondFundForm}) by proving that it converges (when $\e \to 0$) to the exact one when it exists.

\begin{prop}[Consistency of the approximate second fundamental form]\label{prop_pointwise_convergence}
Let $\Omega \subset \R^n$ be an open set and let $V$ be a rectifiable $d$--varifold with bounded variations. Then, for $\V$--almost any $x \in \Omega$ the quantities 
$\beta^{V,\e} (x), c^{V,\e}(x), A^{V,\e} (x), B^{V,\e}$ are defined for $\e$ small enough, and respectively converge to $\beta^V (x), c^V (x), A^V (x), B^V(x)$ as $\e\to 0$.\\
Moreover, there exist constants $C_1 > 0$, $C_2 > 0$ such that
\begin{equation} \label{eqPointwiseCvEstimate}
\left| A^{V,\e} - A^V \right| \leq C_1  | \beta^{V,\e} - \beta^V | + C_2 |H^V| |c^{V,\e} - c^V | \: .
\end{equation}
\end{prop}

\begin{proof} 
The proof of the convergence of $\beta^{V,\e}$ and $c^{V,\e}$ is obtained arguing as in \cite[Proposition $4.1$]{BuetLeonardiMasnou}. For the sake of clarity, we provide it hereafter.
For $x \in \supp\V\cap \Omega$ and $\e>0$ small enough, we can drop $V$ in the exponent for simplicity, and write
\begin{equation} \label{eq:PointwiseCV1}
\left| \beta_{ijk}^\e (x) - \beta_{ijk} (x) \right|  = \frac{C_\xi}{C_\rho}\left|\frac{\Big(\beta_{ijk} \V - \delta_{ijk} V_s \Big) \ast \rho_\e (x)}{\V \ast \xi_\e (x)} - \frac{C_\rho}{C_\xi}\beta_{ijk}(x) \right| \leq  I_1(\e) + I_2(\e) + I_3(\e) \,,
\end{equation}
where 
\begin{align*}
I_1(\e) &= \frac{C_\xi}{C_\rho}\left|\frac{\int \Big(\beta_{ijk}(y) - \beta_{ijk}(x)\Big)\rho_\e(x-y)\, d\V(y)}{\V\ast \xi_\e(x)}\right|\,,\\
I_2(\e) &= \frac{C_\xi}{C_\rho}\cdot |\beta_{ijk}(x)|\cdot \left|\frac{\V\ast \rho_\e(x)}{\V\ast \xi_\e(x)} - \frac{C_\rho}{C_\xi}\right|\,,\\
I_3(\e) &= \frac{C_\xi}{C_\rho}\frac{|\delta_{ijk}V_s \ast \rho_\e(x)|}{\V\ast \xi_\e(x)}\,.
\end{align*}
We recall that $V$ is rectifiable, that is, $V = v(M,\theta)$. Then we preliminarily notice that, for $\V$-almost every $x$, 
\begin{align}\nonumber
\e^{-d} \V (B_\e (x)) &\xrightarrow[\e \to 0]{} \omega_d \theta (x),\\ \label{eq:PointwiseCVdensity}
\epsilon^{n-d} \V \ast \rho_\epsilon (x) &\xrightarrow[\epsilon \to 0]{} C_\rho \theta (x)>0,\\\nonumber
\epsilon^{n-d} \V \ast \xi_\epsilon (x) &\xrightarrow[\epsilon \to 0]{} C_\xi \theta (x)>0\,,
\end{align}
which follows from the definition of the approximate tangent plane and the fact that $\rho \in \xC_c^0 (\R^n)$.
Thanks to \eqref{eq:PointwiseCVdensity}, for $\V$-almost any $x \in \Omega$ (such that $x$ is also a Lebesgue point of $\beta_{ijk} \in \xL^1 (\V)$), we have
\begin{align*}
I_1(\e) &\le \frac{C_\xi\|\rho\|_\infty \e^{-d}\V(B_\e(x))}{C_\rho \e^{n-d}\V\ast \xi_\e(x)} \cdot \frac{1}{\V(B_\e(x))} \int_{B_\e(x)}|\beta_{ijk}(y) - \beta_{ijk}(x)|\, d\V(y) \xrightarrow[\epsilon \to 0]{} 0\,,\\
I_2(\e) &\le \frac{C_\xi |\beta_{ijk}(x)|}{C_\rho}\left|\frac{\V\ast \rho_\e(x)}{\V\ast \xi_\e(x)} - \frac{C_\rho}{C_\xi}\right| \xrightarrow[\epsilon \to 0]{} 0\,,\\ 
I_3(\e) &\le \frac{\left| \delta_{ijk} V_s \right| \ast \rho_\epsilon (x) }{ \V \ast \xi_\epsilon (x) } \leq \frac{C_\xi\|\rho\|_\infty \e^{-d}\V(B_\e(x))}{C_\rho \e^{n-d}\V\ast \xi_\e(x)} \cdot \frac{|\delta_{ijk} V_s|(B_\e(x))}{\V(B_\e(x))} \xrightarrow[\epsilon \to 0]{} 0\,,
\end{align*}
The previous estimate together with \eqref{eq:PointwiseCV1} lead to $\beta_{ijk}^\e (x) \xrightarrow[\e \to 0]{} \beta_{ijk} (x)$, as wanted.

Now we observe that the map $x \mapsto \int_{\G} S \, d \nu_x (P)$ is in $\xL^1 (\V)$, then by a similar proof, where essentially one replaces $\beta_{ijk} (x)$ with $\int_{\G} S \, d \nu_x (P)$, we get the convergence of $c^\e$ to $c$ for $\V$ almost every $x\in \Omega$; and the convergences of $A^{V,\e}$ and $B^{V,\e}$ follow.

We are thus left with the proof of \eqref{eqPointwiseCvEstimate} which follows from the fact that $|c_{jk}^{V,\e}| \leq 1$ and that $\| (I+c^{V,\e})^{-1} \|$ and $\| (I+c^{V})^{-1} \|$ are uniformly bounded by a dimensional constant $C_0$ (see \eqref{eqUniformBoundAijkBetaijk}). Indeed, from \eqref{aijkexplicit}, we have
\begin{align}
\left| A_{ijk}^{V} - A_{ijk}^{V,\e} \right| & \leq \left| \beta_{ijk}^{V} - \beta_{ijk}^{V,\e}  \right|  + \left| (I + c^V)^{-1} H^V \right| \left| c_{jk}^{V, \e} - c_{jk}^{V} \right| \nonumber \\
& + \left\| c_{jk}^{V, \e} \right\|  \left| (I + c^{V, \e})^{-1} H^{V, \e} - (I + c^V)^{-1} H^V \right| \nonumber \\
& \leq \left| \beta^{V} - \beta^{V,\e}  \right| + C_0 \left| H^V \right| \left| c_{jk}^{V, \e} - c_{jk}^{V} \right| +  \left| (I + c^{V, \e})^{-1} H^{V, \e} - (I + c^V)^{-1} H^V \right| \label{eqPointwiseCvEstimate2}
\end{align}
And for the last term in \eqref{eqPointwiseCvEstimate2}, let us temporarily denote $u = (I + c^V)^{-1} H^V$ and $u^{\e} = (I + c^{V, \e})^{-1} H^{V, \e}$ so that
\[
(I + c^{V, \e}) (u^\e - u) + (c^V - c^{V, \e}) u = (I + c^{V, \e}) u^\e - (I + c^{V}) u = H^{V, \e} - H^V
\]
leading to
\begin{align}
\left|u^\e - u \right| & = \left| (I + c^{V, \e})^{-1} \left( H^{V, \e} - H^V - (c^V - c^{V, \e}) u \right) \right| \nonumber \\
& \leq C_0 \left( \left| H^{V, \e} - H^V \right| + \left\| c^V - c^{V, \e} \right\| \left| (I + c^V)^{-1} H^V \right| \right) \nonumber \\
& \leq C_0 n \left| \beta^{V} - \beta^{V,\e}  \right| + C_0^2 \left| H^V \right| \left\| c^V - c^{V, \e} \right\| \label{eqPointwiseCvEstimate3}
\end{align}
since $\displaystyle H_i^V = \sum_q \beta_{qiq}^V$ and $\displaystyle H_i^{V, \e} = \sum_q \beta_{qiq}^{V, \e}$. Plugging \eqref{eqPointwiseCvEstimate3} into \eqref{eqPointwiseCvEstimate2} we obtain \eqref{eqPointwiseCvEstimate}.

\end{proof}

\subsection{Convergence of the approximate second fundamental form of a sequence of varifolds}
\label{sectionConvergence}

We now turn to the following question: given a sequence of $d$--varifolds $(V_h)_h$ approximating (in the sense of weak--$\ast$ convergence for instance) a $d$--rectifiable varifold with bounded variations, does the $\e$--WSFF $B^{V_h,\e_h}$ of $V_h$ converge to the WSFF $B^V$ of $V$ under suitable assumptions? In order to answer this question we follow the scheme of \cite{BuetLeonardiMasnou} and obtain two convergence results. A first one showing a slower convergence rate, but requiring weaker assumptions; a second one showing a better convergence rate, but under stronger assumptions. The proofs are slight variants of the ones given in \cite[Theorems $4.5$ and $4.8$]{BuetLeonardiMasnou}), therefore we shall only provide the key estimates and underline the main differences with\cite{BuetLeonardiMasnou}.
 
\begin{theo} \label{theoConvergence1}
Let $\Omega \subset \R^n$ be an open set and $V$ be a rectifiable $d$--varifold with bounded variations. Let $(V_h)_h$ be a sequence of $d$--varifolds and let $(\eta_h)_h, (d_h)_h$ be two positive, decreasing and infinitesimal sequences satisfying the following property: for any ball $B \subset \Omega$, with small enough radius and centered in $\supp \V$, one has
\begin{equation} \label{eqFlatDistanceControl}
\Delta_B (V,V_h) \leq d_h \min \left( \V (B^{\eta_h}) , \| V_h \| (B^{\eta_h}) \right). 
\end{equation}
Then, for $\V$--almost every $x$, for any sequence $(z_h)_h$ tending to $x$, and for any infinitesimal sequence $(\epsilon_h)_h$ such that 
\[
\frac{\eta_h}{\epsilon_h} \xrightarrow[h \to \infty]{} 0\,,
\]
one has for $h$ large enough
\begin{equation} \label{eqCVMatrix}
\begin{array}{rl}
\displaystyle
\max_{i,j,k} \left\lbrace \left| \beta^{V_h,\e_h}_{ijk} (z_h) - \beta^{V,\e_h}_{ijk} (x) \right| ,  \left| A^{V_h,\e_h}_{ijk} (z_h) - A^{V,\e_h}_{ijk} (x) \right| \right\rbrace & \leq C\| \rho \|_{\xW^{2,\infty}} \displaystyle \frac{d_h + |x-z_h|}{\epsilon_h^2} , \\[12pt]
 \left| c^{V_h,\e_h} (z_h) - c^{V,\e_h} (x) \right| \leq C\| \rho \|_{\xW^{1,\infty}} \displaystyle \frac{d_h + |x-z_h|}{\epsilon_h} .
\end{array}
\end{equation}
In particular, both right-hand sides of \eqref{eqCVMatrix} are infinitesimal as soon as 
\[
\frac{d_h + |x-z_h|}{(\epsilon_h)^2} \xrightarrow[h \to \infty]{} 0\,.
\]
\end{theo}

\begin{proof}
In order to prove \eqref{eqCVMatrix}, it is enough to prove that the right-hand sides of \eqref{eqCVMatrix} provide upper bounds for, respectively, $\left| \beta^{V_h,\e_h}(z_h) - \beta^{V,\e_h}(x) \right|$ and $|c^{V_h,\e_h}(z_h) - c^{V,\e_h}(x)|$. The proof is exactly the same as that of Theorem 4.5 in \cite{BuetLeonardiMasnou}. Shortly, we combine Proposition~\ref{prop_pointwise_convergence} with the two estimates below: for $\epsilon > 0$,
\begin{multline} \label{eqLemma4.4_1}
 \left| \int S \rho \left( \frac{y-z_h}{\epsilon} \right) \,  dV_h (y,S)  - \int S \rho \left( \frac{y-x}{\epsilon} \right) \, dV(y,S) \right| \\
  \leq \frac{1}{\epsilon} \| \rho \|_{\xW^{1,\infty}} \left( \Delta_{B_{\epsilon + |x-z_h |}(x) } (V_h,V) + |x-z_h| \V \left( B_{\epsilon + |x-z_h |}(x) \right) \right)
\end{multline}
and 
\begin{multline} \label{eqLemma4.4_2}
 \epsilon^n \left| \delta_{ijk} V_h \ast \rho_\epsilon (z_h)  - \delta_{ijk} V \ast \rho_\epsilon (x) \right| \\ = \frac{1}{\epsilon} \left|  \int S_{jk} \nabla^S \rho \left( \frac{y-z_h}{\epsilon} \right) \, dV_h (y,S) -  \int S_{jk} \nabla^S \rho \left( \frac{y-x}{\epsilon} \right) \, dV (y,S) \right| \\  \leq \frac{1}{\epsilon^2} \| \rho \|_{\xW^{2,\infty}} \left( \Delta_{B_{\epsilon + |x-z_h |}(x) } (V_h,V) + |x-z_h| \V \left( B_{\epsilon + |x-z_h |}(x) \right) \right).
\end{multline}
In order to prove \eqref{eqLemma4.4_1} and \eqref{eqLemma4.4_2}, we first observe that following, general inequality holds:
\begin{align}
\left| \int \psi_h (y,S) \, dV_h - \int \psi (y,S) \, dV \right| & \leq \left| \int \psi_h (y,S) \, d (V_h - V) \right| + \left| \int \left(\psi_h (y,S) - \psi(y,S) \right) \, dV \right| \nonumber \\
& \leq \xlip (\psi_h) \Delta_{B_h} (V_h,V) +  \int \left| \psi_h (y,S) - \psi(y,S) \right| \, dV . \label{eqProofLemma4.4_1}
\end{align}
where $B_h$ contains $\supp \psi_h$. 
In the case of \eqref{eqLemma4.4_1}, we have
\[
\psi_h (y,S) = S \rho \left( \frac{y-z_h}{\e_h} \right), \quad \xlip(\psi_h) = \frac{1}{\e_h}  \| \rho \|_{\xW^{1,\infty}}, \quad \supp \psi_h \subset B_{\e_h + |x-z_h|}(x),
\]
 therefore we can take $B_h = B_{\e_h + |x-z_h|}(x)$ and $\displaystyle \psi (y,S) = S \rho \left( \frac{y-x}{\e_h} \right)$, so that plugging
\[
\int \left| \psi_h (y,S) - \psi(y,S) \right| \, dV \leq \frac{1}{\e_h}\xlip(\rho) \int_{B_{\e_h + |x-z_h|}(x)} |x-z_h| \, dV \leq \xlip(\rho) \frac{|x-z_h|}{\e_h} \V \left( B_{\e_h + |x-z_h|}(x) \right)
\]
into \eqref{eqProofLemma4.4_1} gives \eqref{eqLemma4.4_1}. In the case of \eqref{eqLemma4.4_2} we have
\[
\psi_h (y,S) = \frac{1}{\e_h} S_{jk} \nabla^S \rho \left( \frac{y-z_h}{\e_h} \right), \quad \xlip(\psi_h) = \frac{1}{\e_h^2}  \| \rho \|_{\xW^{2,\infty}}, \quad  \supp \psi_h \subset B_{\e_h + |x-z_h|}(x),
\] therefore we can still take $B_h = B_{\e_h + |x-z_h|}(x)$ and set $\displaystyle \psi (y,S) = \frac{1}{\e_h} S_{jk} \nabla^S \rho \left( \frac{y-x}{\e_h} \right)$, so that plugging 
\begin{align*}
\int \left| \psi_h (y,S) - \psi(y,S) \right| \, dV & \leq \frac{1}{\e_h^2}\xlip(\nabla \rho) \int_{B_{\e_h + |x-z_h|}(x)} |x-z_h| \, dV \\
& \leq \xlip(\nabla \rho) \frac{|x-z_h|}{\e_h^2} \V \left( B_{\e_h + |x-z_h|}(x) \right)
\end{align*}
into \eqref{eqProofLemma4.4_1} gives \eqref{eqLemma4.4_2}.
\end{proof}

The convergence estimate \eqref{eqCVMatrix} proved in  Theorem~\ref{theoConvergence1}, and obtained under quite general assumptions, shows an asymptotic rate proportional to $\frac{d_h}{\epsilon_h^2}$. On the other hand, we can obtain a better convergence rate under more restrictive assumptions, as shown in Theorem~\ref{theoConvergence2} below. The idea is that, whenever $M$ is a smooth $d$--submanifold of $\R^n$ and $V=v(M,1)$, the mean curvature vector is orthogonal to the tangent plane, therefore we have 
\begin{align}
\int S_{jk} & \, \nabla^S \rho \left( \frac{y-x}{\epsilon} \right) \, dV(y,S)  = \int_{y \in M} P(y)_{jk}\, \nabla^{P(y)} \rho \left( \frac{y-x}{\epsilon} \right) \, d \cH^d(y) \nonumber \\
& = \int_{y \in M} \underbrace{\left[ P(y)_{jk} - P(x)_{jk} \right] \nabla^{P(y)} \rho \left( \frac{y-x}{\epsilon} \right)}_{= \phi(y)} \, d \cH^d(y) + P(x)_{jk} \int_{y \in M} \nabla^{P(y)} \rho \left( \frac{y-x}{\epsilon} \right) \, d \cH^d(y) \label{eqOrthogDecomposition}
\end{align}
In \eqref{eqOrthogDecomposition}, as soon as the tangent plane is Lipschitz, we will gain one factor of order $\epsilon$ in the first term when computing the Lipschitz constant of $\phi$, which leads to the same gain of order when estimating the pointwise error $\displaystyle \left| \beta^{V_h,\e_h} (z_h) - \beta^{V,\e_h} (x) \right|$, as seen in the proof of Theorem~\ref{theoConvergence1}. Concerning the second term, we know from Theorem $4.8$ in \cite{BuetLeonardiMasnou} that we gain an order $\epsilon$ by projecting the tangential gradient to the orthogonal space at $x$. When defining a modified version of $\beta^V$
in accordance with \eqref{eqOrthogDecomposition}, we recover the $\beta^{V, \perp}$ introduced in Definition~\ref{dfnBetaAijkPerp}. We next propose an approximate version of $\beta^{V, \perp}$:

\begin{dfn}[Approximate orthogonal weak second fundamental form, $\epsilon$--WSFF$^\perp$] \label{dfnBetaAijkPerpReg}
Let $\Omega \subset \R^n$ be an open set and let $V = \V \otimes \nu_x$ be a $d$--varifold in $\Omega$. For $x \in \Omega$, we define
\begin{equation}
\beta_{ijk}^{V, \e, \perp} (x) = \beta_{ijk}^{V,\e} (x) - c_{jk}^V (x) \left( c^V(x) \sum_q \beta_{q \cdot q}^{V,\e} (x) \right)_i.
\end{equation}
We call approximate orthogonal weak second fundamental form (referred to as $\epsilon$--WSFF$^\perp$) the tensor $A^{V, \e, \perp} (x)=\{A_{ijk}^{V, \e, \perp} (x)\}$ where each $A_{ijk}^{V, \e, \perp} (x)$ is the solution to the linear system \eqref{eq:mainSystem2} with $b_{ijk} = \beta_{ijk}^{V, \perp, \e} (x)$ and $c_{jk} = c_{jk}^V(x)$. With a slight abuse of terminology, we will also call approximate orthogonal weak second fundamental form the tensor $\{B_{ij}^{k,\e, \perp}\}$ defined by $\displaystyle B_{ij}^{k,\e, \perp} = \frac{1}{2} ( A_{ijk}^{V,\e, \perp} + A_{jik}^{V,\e, \perp} - A_{kij}^{V,\e, \perp} )$  (see~\eqref{eq:explicitExpressionB}).
\end{dfn}

Notice that by \eqref{aijkexplicit} and following the computations in the proof of Proposition~\ref{propStructuralPropertiesPerp}, we get that
\begin{equation} \label{eqAijkPerpExplicit}
A_{ijk}^{V, \perp, \e}  = \beta_{ijk}^{V,\perp, \e} - c^V_{jk} \left( (I + c^V)^{-1} \sum_{q=1}^n \beta_{q\cdot q}^{V, \perp, \e} \right)_i 
 = \beta_{ijk}^{V, \e} - c^V_{jk} \sum_{q=1}^n \beta_{qiq}^{V, \e} \, .
\end{equation}

\begin{remk}
Notice that we use $c^V(x)$ rather than the regularized expression $c^{V, \e}(x)$ in Definition~\ref{dfnBetaAijkPerpReg}. If we used $c^{V, \e}(x)$, we would directly recover the convergence of Theorem~\ref{theoConvergence1}. Yet, under the assumptions of Theorem~\ref{theoConvergence2}, the regularization of $c^V$ is not necessary anymore and that is why we drop it in Definition~\ref{dfnBetaAijkPerpReg}.
\end{remk}

\begin{theo} \label{theoConvergence2}
Let $\Omega \subset \R^n$ be an open set, $M \subset \Omega$ be a $d$--dimensional submanifold of class $\xC^2$ without boundary, and let $V = v(M,1)$. Let $P$ be a $\xC^1$ extension of the tangent map $T_y M$ on a tubular neighborhood of $M$. Let $(V_h)_h$ be a sequence of $d$--varifolds in $\Omega$. Choose $x \in M$ and a sequence $(z_h)_h \subset \Omega$ converging to $x$, such that $z_h \in \supp \| V_h\|$. Assume that there exist positive, decreasing and infinitesimal sequences $(\eta_h)_h, \,(d_{1,h})_h, \, (d_{2,h})_h, \, (\e_h)_h $, such that for any ball $B\subset\Omega$ centered in $\supp \V$ and contained in a neighborhood of $x$, one has
\begin{equation} \label{eq_thm_pointwise_cvSmooth_hyp2}
\Delta_B(\V,\|V_h\|) \leq d_{1,h} \min \left( \V(B^{\eta_h}) , \|V_h\|(B^{\eta_h}) \right) \: ,
\end{equation}
and, recalling the decomposition $V_h = \| V_h \| \otimes \nu_x^h$,
\begin{equation} \label{eq_thm_pointwise_cvSmooth_hyp3}
\sup_{ \{y \in B_{\e_h + |x-z_h|} (x) \cap \supp  \| V_h \| \} } \int_{S \in \G} \| P(y) - S \| \, d \nu_y^h (S) \leq d_{2,h} \: .
\end{equation}
Then, there exists $C >0$ such that for $h$ large enough,
\begin{equation}\label{eq:betaAperp}
\max\left\{\left| \beta_{ijk}^{V_h,\e_h,\perp} (z_h) - \beta_{ijk}^{V,\e_h,\perp}(x) \right|, \left| A_{ijk}^{V_h,\e_h,\perp} (z_h) - A_{ijk}^{V,\e_h,\perp}(x) \right|\right\} \leq C \frac{d_{1,h} + d_{2,h} + |x-z_h|}{\e_h} 
\end{equation}
for all $i,j,k$. 
Moreover, if we also assume that $d_{1,h}+d_{2,h}+\eta_{h}+|x-z_{h}| = o(\e_{h})$ as $h\to\infty$, then
\begin{equation*}
\beta_{ijk}^{V_h,\e_h,\perp} (z_h) \xrightarrow[h \to \infty]{} \beta_{ijk}^{V}(x) \,.
\end{equation*}
\end{theo}

\begin{proof}
First of all, note that as for the proof of Theorem~\ref{theoConvergence1}, we only need to bound $| c^{V_h}(z_h) - c^{V}(x) |$ and $\left| \beta_{ijk}^{V_h,\e_h,\perp}(z_h) - \beta_{ijk}^{V,\e_h,\perp}(x) \right|$ from above by the right-hand side of \eqref{eq:betaAperp}, and the upper bound for $|c^{V_h}(z_h) - c^{V}(x)|$ is a direct consequence of \eqref{eq_thm_pointwise_cvSmooth_hyp3}. Let us rewrite explicitly $\beta_{ijk}^{V, \e, \perp}$:
\begin{multline} \label{eqOrthogBeta}
\beta_{ijk}^{V,\e, \perp} (x) = -\frac{C_\xi}{C_\rho} \frac{1}{\V \ast \xi_\epsilon(x)}  \int \Big( S_{jk} - \int_{P \in \G} P_{jk} \, d \nu_x(P) \Big) \left( \nabla^S \rho_\e(y-x) \right)_i  \, dV(y,S) \\
 -  \int_{P \in \G} P_{jk} \, d \nu_x(P) \underbrace{ \frac{C_\xi}{C_\rho} \frac{1}{\V \ast \xi_\epsilon(x)}  \int_{P \in \G} \int  \left( P^\perp\nabla^S \rho_\e(y-x) \right)_i \, dV(y,S) \, d \nu_x(P) }_{=:-H_i^{V,\e,\perp}(x)}  \: .
\end{multline}
We know from Theorem $4.8$ in \cite{BuetLeonardiMasnou} that for $h$ large enough
\begin{multline} \label{eqEstimateHperp}
a_h := \left| \int_{P} \int  \frac{ \left( P^\perp\nabla^S \rho_{\e_h}(y-x) \right)_i}{\V \ast \xi_{\e_h}(x)} \, dV(y,S) \, d \nu_x(P) \hspace{-2pt} - \hspace{-2pt} \int_{P} \int  \frac{ \left( P^\perp\nabla^S \rho_{\e_h}(y-z_h) \right)_i }{\|V_h\| \ast \xi_{\e_h} (z_h)} \, dV_h (y,S) \, d \nu_{z_h}(P) \right| \\ \leq  C \frac{d_{1,h} + d_{2,h} + |x-z_h|}{\e_h} \: .
\end{multline}

\noindent
In the following we fix $j,k$ and prove that
\begin{multline}\label{eq:stimastep2}
b_h:=\frac{1}{\e_h}\left| \int \left[ S_{jk} - \int_{P \in \G} P_{jk} \, d \nu_{z_h}(P) \right] \nabla^S \rho \left( \frac{y-z_h}{\e_h} \right) \, dV_h (y,S) \right. \\
\left. - \int \left[ S_{jk} - \int_{P \in \G} P_{jk} \, d \nu_x(P) \right] \nabla^S \rho \left( \frac{y-x}{\e_h} \right) \, dV(y,S) \right|\\
 \leq C \frac{d_{1,h}+d_{2,h}+|x-z_h|}{\e_h} \| V_h \| \left( B_{\e_h + |x-z_h| + \eta_h}(x) \right).
\end{multline}
where $C$ is a constant depending only on $\xlip(P)$ and $\|\rho\|_{\xW^{2,\infty}}$. To this aim we write
\begin{multline}\label{eq:stimastep2bis}
\frac{1}{\e_h}\left| \int \left[ S_{jk} - \int_{P \in \G} P_{jk} \, d \nu_{z_h}(P) \right] \nabla^S \rho \left( \frac{y-z_h}{\e_h} \right) \, dV_h (y,S) \right. \\
\left. - \int \left[ S_{jk} - \int_{P \in \G} P_{jk} \, d \nu_x(P) \right] \nabla^S \rho \left( \frac{y-x}{\e_h} \right) \, dV(y,S) \right|
\leq \frac{c_h + d_h + e_h}{\e_h}\,,
\end{multline} 
where $c_h$, $d_h$ and $e_h$ are defined - and estimated - hereafter. We start with 
\begin{equation*}
c_h := \left| \int \psi_h (x,y) \, d \|V_h\|(y) - \int \psi_h(x,y) \, d \V(y) \right|\,,
\end{equation*}
where $\displaystyle \psi_h(x,y) = \left( P(y)_{jk} - P(x)_{jk} \right) \nabla^{P(y)} \rho \left( \frac{y-x}{\e_h} \right)$, so that $\psi_h$ is Lipschitz, with $\xlip(\psi_h)$ depending only on $\xlip(P)$ and on $\| \rho \|_{\xW^{2,\infty}}$. Note also that $\supp \psi_h(x,\cdot) \subset B_{\epsilon}(x)$. Therefore we obtain
\begin{equation}
c_h \leq \xlip(\psi_h) \Delta_{B_{\e_h}(x)} \left( \| V_h \| , \V \right) \leq \xlip(\psi_h) d_{1,h} \| V_h \| \left( B_{\e_h + \eta_h} (x) \right).
\end{equation}
Then we have
\[
d_h : = \left| \int \Big(\psi_h(z_h,y) - \psi_h(x,y)\Big) \, dV_h (y,S)  \right| 
 \leq \xlip(\psi_h) |x-z_h| \| V_h \|\left( B_{\e_h + |x-z_h|}(x) \right)\,.
\]
Finally we have
\begin{multline*}
e_h :  = \left| \int \left[ \left( S_{jk} - \int_{P \in \G} P_{jk} \, d \nu_{z_h}(P) \right) \nabla^S \rho \left( \frac{y-z_h}{\e_h} \right) - \right.\right. \\
 \left.\left. \left( P(y)_{jk} - P(z_h)_{jk} \right)  \nabla^{P(y)} \rho \left( \frac{y-z_h}{\e_h} \right) \right] \, dV_h (y,S)  \right| \\
 \leq  \left| \int \left[ \left( S_{jk} - \int_{P \in \G} P_{jk} \, d \nu_{z_h}(P) \right) -  \left( P(y)_{jk} - P(z_h)_{jk} \right) \right] \nabla^S \rho \left( \frac{y-z_h}{\e_h} \right) \, dV_h (y,S)  \right| \\
 + \left| \int   \left( P(y)_{jk} - P(z_h)_{jk} \right) \left( \Pi_S - \Pi_{P(y)} \right) \left[ \nabla \rho \left( \frac{y-z_h}{\e_h} \right) \right] \, dV_h (y,S)  \right| \\
 \leq  4 d_{2,h} \| \rho \|_{\xW^{1,\infty}} \| V_h \| \left( B_{\e_h + |z_h-x|}(x) \right)\,,
\end{multline*}
so that by combining these estimates with \eqref{eq:stimastep2bis} we obtain \eqref{eq:stimastep2}. 

\noindent
Eventually, we know from \eqref{eqOrthogBeta} that
\begin{multline} \label{eqOrthogBetaFin}
\left| \beta_{ijk}^{V_h, \e, \perp} (z_h) - \beta_{ijk}^{V, \e, \perp} (x)  \right| \leq \frac{C_\xi}{C_\rho} \left( a_h + \frac{b_h}{(\e_h)^n \|V_h\| \ast \xi_{\e_h} (z_h)} \right) \\
+ \frac{\left| \beta_{ijk}^{V, \e_h, \perp} (x) - P_{jk}(x) H_i^{V,\e_h,\perp} (x) \right|}{ \|V_h\| \ast \xi_{\e_h} (z_h)} \left| \V \ast \xi_{\e_h} (x) - \| V_h \| \ast \xi_{\e_h} (z_h)  \right| \: .
\end{multline}
The conclusion follows from injecting \eqref{eqEstimateHperp} and \eqref{eq:stimastep2} in \eqref{eqOrthogBetaFin}, combined with Lemma $4.4$ in \cite{BuetLeonardiMasnou} and the fact that $\left| \beta_{ijk}^{V, \e_h, \perp} (x) - P_{jk}(x) H_i^{V,\e_h,\perp} \right| \xrightarrow[h \to \infty]{} \left| \beta_{ijk}^{V, \perp} (x) - P_{jk}(x) H_i^{V,\perp} (x) \right|$ is thus bounded. The last part of the statement follows from \eqref{eq:betaAperp} and Proposition~\ref{propBetaPerp}.
\end{proof}

\section{Implementation and numerical illustrations}
\label{sectionNumerics}

We present in this section a few numerical illustrations of approximate curvatures for point clouds. We recall for more convenience the notion of point cloud varifold (see Definition \ref{def:PCV}). Given a finite set of points $\{ x_l \}_{l=1 \ldots N} \subset \R^n$, of "masses" $\{ m_l \}_{l = 1\ldots N} \subset \R_+$, and $d$--planes $\{ P_l \}_{l = 1 \ldots N} \subset \G$, we associate a $d$--varifold $V$ in $\R^n$ defined as
\begin{equation} \label{eqPCV}
V = \sum_{l=1}^N m_l \delta_{(x_l, P_l)} \: .
\end{equation}
While the variations $\delta_{ijk} V$ of $V$ are generally not Radon measures but distributions of order $1$, it is yet possible to compute explicitly the $\e$--WSFF of $V$ for any $\epsilon>0$, as explained in Section~\ref{Section:RegWSFF}. Explicit formulas for point cloud varifolds are derived in Section~\ref{SectionNumericsWFFPointCloud} below, while the implementation is detailed in Section~\ref{SectionImplementation}. Finally, in Section~\ref{SectionNumericalExp} we perform some numerical experiments on $2$--dimensional point clouds in $\R^3$.

\subsection{Approximate WSFF$^\perp$ for point clouds}
\label{SectionNumericsWFFPointCloud}

Let us give the precise expression of the $\e$--WSFF$^\perp$ in the case of the point cloud varifold $V$ defined in \eqref{eqPCV}.  We take, up to re-normalization, the profile $\rho(t) = \exp \left( - \frac{1}{1 - t^2} \right)$ for $t \in [0,1)$ and then define $\xi$ according to \eqref{eq:NKP}. Let $\e > 0$ and $l_0 \in \{ 1, \ldots, N \}$. We have
\begin{equation*} 
\beta_{ijk}^{V,\e} (x_{l_0}) = \frac{d}{n} \frac{\displaystyle  \sum_{l=1}^N m_l \left( P_l \right)_{jk} \rho^\prime \left( \frac{| x_{l_0} - x_l|}{\e} \right) \frac{P_l (x_{l_0} - x_l)}{|x_{l_0} - x_l|} \cdot e_i}{\displaystyle \sum_{l=1}^N m_l \xi  \left( \frac{| x_{l_0} - x_l|}{\e} \right)} \, .
\end{equation*}
Then, thanks to  $c^V (x_{l_0}) = P_{l_0}$ and \eqref{eqAijkPerpExplicit}, $A_{ijk}^{V, \perp, \e} (x_{l_0})$ is defined as
\begin{equation*}
A_{ijk}^{V, \perp, \e} (x_{l_0}) = \frac{d}{n} \frac{\displaystyle  \sum_{l=1}^N m_l \left( \left( P_l \right)_{jk} - \left( P_{l_0} \right)_{jk} \right) \rho^\prime \left( \frac{| x_{l_0} - x_l|}{\e} \right) \frac{P_l (x_{l_0} - x_l)}{|x_{l_0} - x_l|} \cdot e_i}{\displaystyle \sum_{l=1}^N m_l \xi  \left( \frac{| x_{l_0} - x_l|}{\e} \right)} \, .
\end{equation*}
Lastly, by Definition~\ref{dfnBetaAijkPerpReg}:
\begin{equation} \label{eqExplcitSecondFundamentalFormPCL}
B_{ij}^{k, \perp, \e} (x_{l_0}) =   \frac{\displaystyle \frac{d}{n} \sum_{l=1}^N m_l \rho^\prime \left( \frac{| x_{l_0} - x_l|}{\e} \right) \frac{P_l (x_{l_0} - x_l)}{|x_{l_0} - x_l|} \cdot \frac 12 \left(  \left( P_l - P_{l_0} \right)_{jk}  e_i +   \left( P_l - P_{l_0} \right)_{ik}  e_j -  \left( P_l - P_{l_0} \right)_{ij}  e_k \right) }{\displaystyle \sum_{l=1}^N m_l \xi  \left( \frac{| x_{l_0} - x_l|}{\e} \right)}  \, .
\end{equation}
In order to rewrite this latter tensor in a more standard way, we first project, for fixed $i,j$, the vector $(B_{ij}^k)_{k=1 \ldots n}$ onto the normal space by taking:
\[
(I - P_{l_0}) B_{ij}^{\cdot , \perp, \e} (x_{l_0})
\]
and in the case where $d = n-1$ (codimension $1$), we obtain a scalar value by taking the scalar product with a unit vector $n_{l_0}$ generating the normal space. We obtain the matrix $(B_{ij}^{\perp, \e})_{i,j = 1 \ldots n}$:
\[
B_{ij}^{\perp, \e} (x_{l_0}) = (I - P_{l_0}) B_{ij}^{\cdot , \perp, \e} (x_{l_0}) \, \cdot \, n_{l_0} = B_{ij}^{\cdot , \perp, \e} (x_{l_0}) \, \cdot \, n_{l_0} \, .
\]
In order to recover the principal curvatures, it remains to pass from the extended second fundamental form $B_{ij}^{\perp, \e}$ to the one restricted to the tangent space. It can be done by simply considering the $n\times d$ matrix $Q_{l_0}$ whose columns constitute an orthonormal basis of the tangent space at $x_{l_0}$, and then computing
\[
\overline{B_{ij}}^\e (x_{l_0}) = Q_{l_0}^t B_{ij}^{\perp, \e} (x_{l_0}) Q_{l_0}
\]
The eigenvalues and eigenvectors of $\overline{B_{ij}}^\e (x_{l_0})$ are (approximations of) the principal curvatures and directions at $x_{l_0}$.

Notice that as long as no global orientation of the cloud is given, $n_{l_0}$ is chosen up to a sign, therefore the sign of each eigenvalue cannot be recovered, but their relative sign is known. This is perfectly normal as the eigenvalues depend on a chosen orientation. In dimension $d = 2$, it is however possible to compute the Gaussian curvature for instance since it does not depend on the orientation (because being $\kappa_1$, $\kappa_2$ the principal curvatures, one has  $ \kappa_1\kappa_2 = (-\kappa_1) (-\kappa_2)$).

\subsection{Implementation}
\label{SectionImplementation}

The implementation relies on the computation of $\epsilon$--neighborhoods of points in point clouds (or $k$--neighborhoods when the number of neighboring points is fixed rather than the radius of the neighboring ball). This can be done very efficiently in practice with the C++ library \texttt{Nanoflann}\cite{nanoflann}. Linear algebra computations are done with the \texttt{Eigen} library~\cite{eigen}. We use \texttt{CloudCompare}~\cite{cloudcompare} for point clouds rendering.

The approximation of the curvature tensor requires the computation of the masses $(m_l)_l$ and the approximate tangent planes $(P_l)_l$ from the positions $\{x_l\}_l$. As usual, the planes $(P_l)_l$ can be computed by local weighted regression: given a point $x_i$ in the cloud, and given a parameter $\sigma > 0$ (or a number of points $k_\sigma$), we first compute the barycenter $\bar{x}$ of the points contained in the $\sigma$--neighborhood $B_\sigma(x_i)$ of $x_i$ (or of $k_\sigma$ closest neighbors of $x_i$).
Then, with the notation $x = ( x^{(1)}, \ldots, x^{(n)} )$ for the $n$--components of $x\in \R^n$, we compute the $n\times n$ covariance matrix at $x_i$ defined as
$C =(C_{kl})_{k,l=1,\ldots n}$, where 
\[
C_{kl}= \sum_{j=1}^N \rho \left( \tfrac{|x_j - x_i|}{\sigma}\right) 
 \left( x_j^{(k)} - {\bar x}^{(k)} \right) \left( x_j^{(l)} - {\bar x}^{(l)} \right) \: .
\]
Computing the $d$ eigenvectors associated with the $d$ highest eigenvalues gives a basis of a plane $P_i$ which approximates the tangent plane at $x_i$.

As for the computation of the weights $(m_l)_l$, if the discretization is locally  uniform (at scale $\epsilon$),  then one can simply take all masses equal (to $1$ for instance). Otherwise, we fix a number of points $N_{mass}$ and we compute at $x_i$ the (smallest) radius $r_i$ of the ball $B_{r_i} (x_i)$ containing at least $N_{mass}$ points. Then we define the mass $m_i$ at $x_i$ as $\frac{\omega_d r_i^d}{N_{mass}}$, or simply $r_i^d$ (simplifying the constant in \eqref{eqExplcitSecondFundamentalFormPCL}).

\subsection{Numerical experiments}
\label{SectionNumericalExp}

Figure~\ref{figDragonGauss} shows the approximate Gaussian curvature $\kappa_1 \kappa_2$ of a dragon point cloud with $N = 435\ 545$ points (the point cloud was sampled from a mesh from Stanford 3D scanning repository \url{http://graphics.stanford.edu/data/3Dscanrep/}). The curvature values are represented with colors ranging from blue (negative values) to red (positive values) through white. 

We represent in Figure~\ref{figDragonAbs} the approximation on the same point cloud of the sum of the absolute values of the two principal curvatures $|\kappa_1| + |\kappa_2|$. The values are indicated with colors ranging from blue (low values) to red (high values) through green and yellow.

\begin{figure}[!htbp]
\centering
\includegraphics[width=0.60\textwidth]{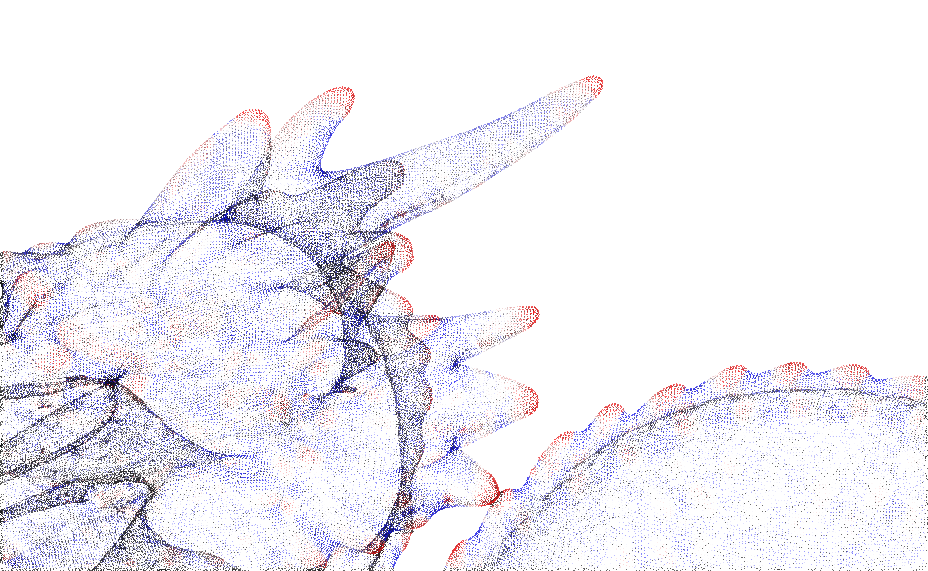}
\includegraphics[width=0.60\textwidth]{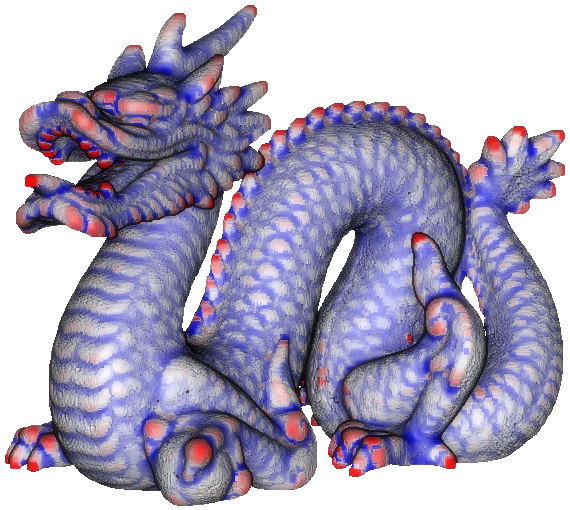}
\caption{Approximate Gaussian curvature $\kappa_1\kappa_2$ of a ``dragon'' point cloud. Top image: detail of the point cloud. Bottom image: the complete point cloud with augmented point size for better rendering. The colors range from blue (negative Gaussian curvature) to red (positive Gaussian curvature) through white. The point cloud has $N = 435\ 545$ points and the approximate Gaussian curvature is computed at each point using $N_{neigh}=40$ nearest points. \label{figDragonGauss}}
\end{figure}

\begin{figure}[!htbp]
\centering
\includegraphics[width=0.60\textwidth]{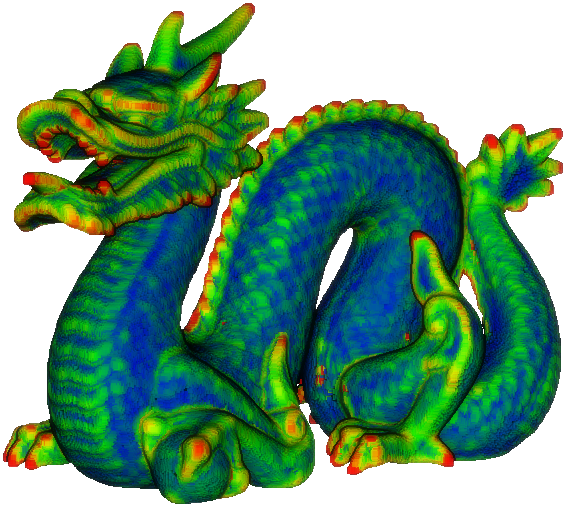}
\caption{Approximation of the sum $|\kappa_1| + |\kappa_2|$  of the absolute values of the principal curvatures of a ``dragon'' point cloud. The colors range from blue (low values) to red (high values) through green and yellow.  The point cloud has $N = 435\ 545$ points and the approximation is computed at each point using $N_{neigh} = 40$ nearest points.\label{figDragonAbs}}
\end{figure}

Figure~\ref{figGenus3Gauss} shows the approximation of various curvature informations (Gaussian curvature, sum of the absolute principal curvatures, and norm of the mean curvature) for a point cloud of $N = 100\,030$ points subsampled on a genus $3$ surface. The number of points used to compute local tangent planes and the $\e$-WSFF$^\perp$ is $N_{neigh} = 40$, corresponding to an average $\e = 0.016$ for a point cloud of diameter approximately $1$.

\begin{figure}[!htbp]
\centering
\subfigure{\includegraphics[width=0.30\textwidth]{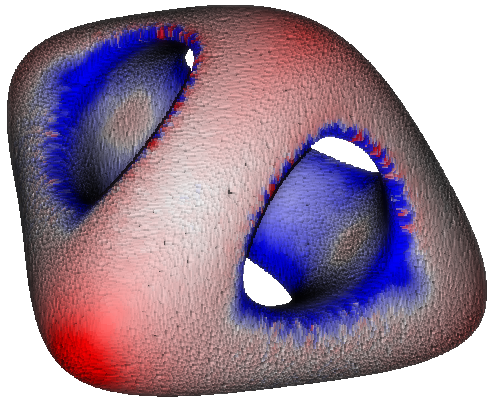}}
\subfigure{\includegraphics[width=0.30\textwidth]{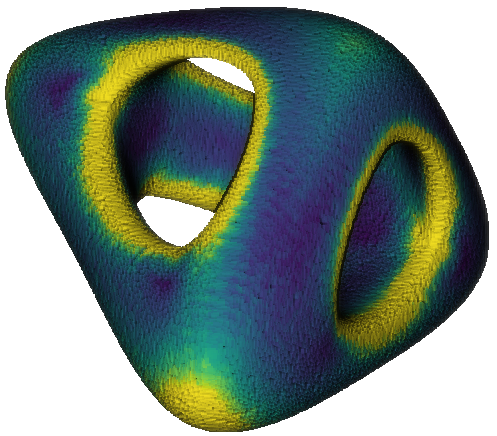}}
\subfigure{\includegraphics[width=0.31\textwidth]{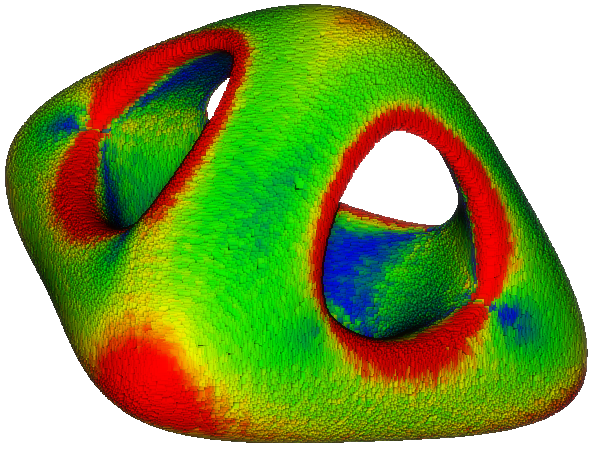}}
\caption{Approximate curvatures of a 3-torus point cloud. From left to right: approximations of the Gaussian curvature $\kappa_1 \kappa_2$ (colors range from blue for negative values to red for positive values, through white), the sum of the absolute principal curvatures $|\kappa_1|+|\kappa_2|$ (colors range from dark blue for low values to yellow for high values, through green), and the norm $|\kappa_1+\kappa_2|$ of the mean curvature vector (colors range from blue for low values to red for high values through green and yellow).\label{figGenus3Gauss}}
\end{figure}

Figure~\ref{pyramid} shows the approximation on a Farman Institute 3D point cloud, see~\cite{farman}, of the sum $|\kappa_1| + |\kappa_2|$ of the absolute principal curvatures at every point of the cloud. This example illustrates that curvature information can be useful for the extraction of features from a dataset.

\begin{figure}[!htbp]
\begin{center}
\includegraphics[width=0.44\textwidth]{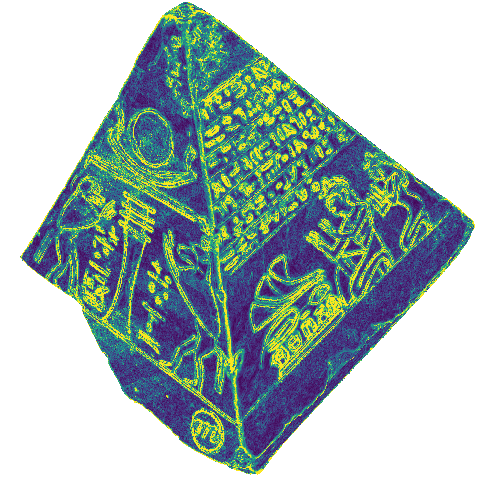}
\caption{Sum of approximate absolute principal curvatures $|\kappa_1| + |\kappa_2|$ computed at each point of a $3D$ point cloud from Farman Institute 3D Point Sets~\cite{farman} (colors range from blue for low values of the sum to yellow and red for high values through green). The meaningful information on the pyramid is clearly associated with high curvature values.\label{pyramid}}
\end{center}
\end{figure}


We conclude with a test performed on a cube discretized with $N = 21\,602$ points (i.e. around $60$ points per edge). The number of points used to compute local tangent planes and the $\e$-WSFF$^\perp$ was $N_{neigh} = 40$, corresponding to an average $\e = 0.060$ for a cube of side--length $1$. The approximate Gaussian curvature is shown in Figure~\ref{figCubeGauss} and the sum of the approximate absolute values of the two principal curvatures $|\kappa_1| + |\kappa_2|$ is represented in Figure~\ref{figCubeAbs}. Figures~\ref{figCubeGaussNoise} and \ref{figCubeAbsNoise} show the results of the same test performed after modifying the positions of all points with an additive centered Gaussian noise of standard deviation $0.01$. A standard deviation equal to $0.05$ was used for Figures~\ref{figCubeGaussNoise2} and \ref{figCubeAbsNoise2}. In these latter experiments, the number of points used to compute local tangent planes and the $\e$-WSFF$^\perp$ was $N_{neigh} = 150$ which corresponds to an average $\e = 0.134$.

\begin{figure}[h]
\subfigure[\label{figCubeGauss}]{\includegraphics[width=0.26\textwidth]{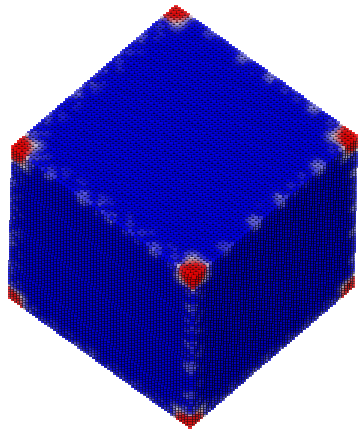}}
\subfigure[\label{figCubeGaussNoise}]{\raisebox{-0.1cm}{\includegraphics[width=0.27\textwidth]{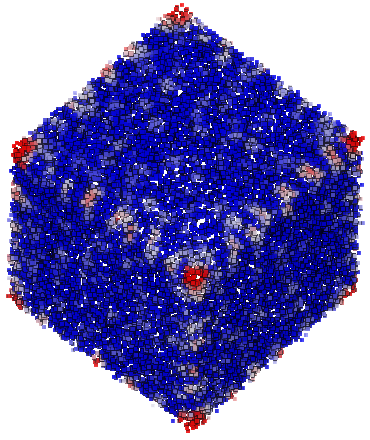}}}
\subfigure[\label{figCubeGaussNoise2}]{\raisebox{-0.4cm}{\includegraphics[width=0.32\textwidth]{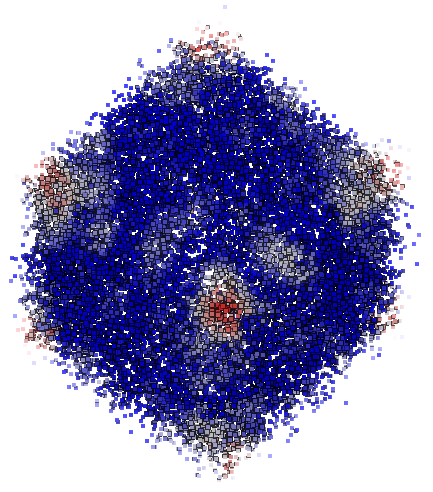}}}
\subfigure[\label{figCubeAbs}]{\includegraphics[width=0.26\textwidth]{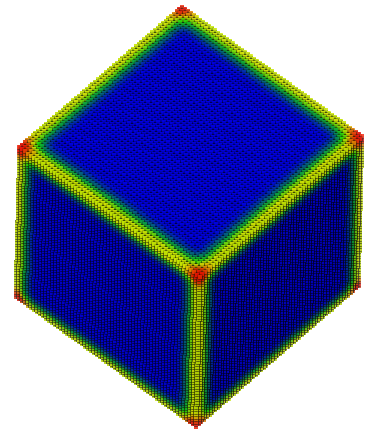}}
\subfigure[\label{figCubeAbsNoise}]{\raisebox{-0.1cm}{\includegraphics[width=0.27\textwidth]{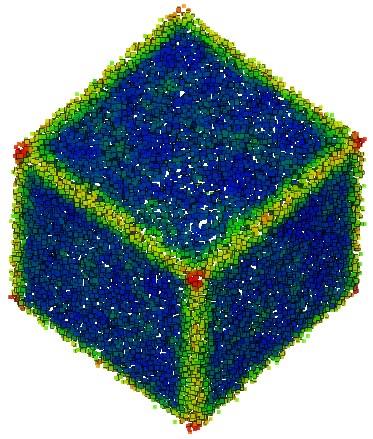}}}
\subfigure[\label{figCubeAbsNoise2}]{\raisebox{-0.4cm}{\includegraphics[width=0.32\textwidth]{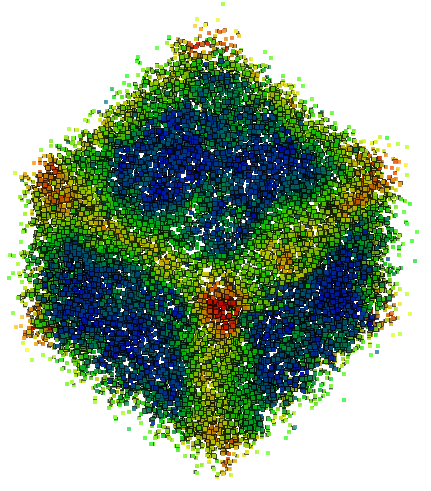}}}
\caption{Testing the stability with respect to noise of the curvatures approximation for a cube point cloud and various levels of Gaussian noise applied to the positions. The level of noise increases from left to right: no noise, additive centered Gaussian noise of standard deviation $0.01$, additive centered Gaussian noise of standard deviation $0.05$. {\bf First row}: $(a)$, $(b)$, $(c)$, Approximate Gaussian curvatures represented with colors ranging from blue (negative values) to red (positive values) through white. {\bf Second row}: $(d)$, $(e)$, $(f)$, Sum of the approximate absolute values of the two principal curvatures $|\kappa_1| + |\kappa_2|$ represented by colors ranging from blue (low values) to red (high values) through green and yellow.}
\end{figure}

\section*{Acknowledgements}
The authors are grateful to Carlo Mantegazza, Ulrich Menne, and Stefan Luckhaus for stimulating discussions about varifolds. They also thank Adriano Pisante for pointing out reference~\cite{Spohn1993}.

\bibliographystyle{alpha}
\newcommand{\etalchar}[1]{$^{#1}$}


\begin{thebibliography}{DPDRG16}

\bibitem[AFP00]{ambrosio}
L.~Ambrosio, N.~Fusco, and D.~Pallara.
\newblock {\em Functions of bounded variation and free discontinuity problems}.
\newblock Oxford mathematical monographs. Clarendon Press, Oxford, New York,
  2000.

\bibitem[AGP98]{ambrosio1998approximation}
L.~Ambrosio, M.~Gobbino, and D.~Pallara.
\newblock Approximation problems for curvature varifolds.
\newblock {\em The Journal of Geometric Analysis}, 8(1):1--19, 1998.

\bibitem[AK04]{Amenta2004}
N.~Amenta and Y.~J. Kil.
\newblock Defining point-set surfaces.
\newblock In {\em ACM SIGGRAPH 2004 Papers}, SIGGRAPH '04, pages 264--270, New
  York, NY, USA, 2004. ACM.

\bibitem[All72]{Allard72}
W.~K. Allard.
\newblock On the first variation of a varifold.
\newblock {\em Annals of mathematics}, pages 417--491, 1972.

\bibitem[All75]{Allard75}
W.~K. Allard.
\newblock On the first variation of a varifold: Boundary behavior.
\newblock {\em Annals of Mathematics}, 101(3):418--446, 1975.

\bibitem[Alm65]{Almgren}
F.~J. Almgren.
\newblock The theory of varifolds, 1965.

\bibitem[AST90]{anzellotti1990curvatures}
G.~Anzellotti, R.~Serapioni, and I.~Tamanini.
\newblock Curvatures, functionals, currents.
\newblock {\em Indiana University Mathematics Journal}, pages 617--669, 1990.

\bibitem[Bel97]{bellettini1997variational}
G.~Bellettini.
\newblock Variational approximation of functionals with curvatures and related
  properties.
\newblock {\em Journal of Convex Analysis}, 4:91--108, 1997.

\bibitem[BLM17]{BuetLeonardiMasnou}
B.~Buet, G.~P. Leonardi, and S.~Masnou.
\newblock A varifold approach to surface approximation.
\newblock {\em Archive for Rational Mechanics and Analysis}, 226:639--694,
  2017.

\bibitem[BM07]{bellettini2007varifolds}
G.~Bellettini and L.~Mugnai.
\newblock A varifolds representation of the relaxed elastica functional.
\newblock {\em Journal of Convex Analysis}, 14(3):543, 2007.

\bibitem[BM10]{bellettini2010approximation}
G.~Bellettini and L.~Mugnai.
\newblock Approximation of {H}elfrich's functional via diffuse interfaces.
\newblock {\em SIAM Journal on Mathematical Analysis}, 42(6):2402--2433, 2010.

\bibitem[Bog07]{bogachev2007}
V.~I. Bogachev.
\newblock {\em Measure Theory II}.
\newblock Springer Berlin Heidelberg, 2007.

\bibitem[Bra78]{brakke}
K.~A. Brakke.
\newblock The motion of a surface by its mean curvature.
\newblock {\em Mathematical notes (20), Princeton University Press}, 1978.

\bibitem[CCLT09]{thibert}
F.~Chazal, D.~Cohen{-}Steiner, A.~Lieutier, and B.~Thibert.
\newblock Stability of curvature measures.
\newblock {\em Computer Graphics Forum}, 28(5), 2009.

\bibitem[CDL03]{DeLellisColding}
T.~Colding and C.~De~Lellis.
\newblock The min--max construction of minimal surfaces.
\newblock {\em Surveys in differential geometry}, 8:75--107, 2003.

\bibitem[CLL14]{Coeurjolly2014}
D.~Coeurjolly, J.-O. Lachaud, and J.~Levallois.
\newblock Multigrid convergent principal curvature estimators in digital
  geometry.
\newblock {\em Computer Vision and Image Understanding}, 129:27 -- 41, 2014.
\newblock Special section: Advances in Discrete Geometry for Computer Imagery.

\bibitem[Clo]{cloudcompare}
CloudCompare.
\newblock A {3D} point cloud and mesh processing software.
\newblock \url{http://www.danielgm.net/cc/}.

\bibitem[CSM06]{morvan_cohen_steiner}
D.~Cohen-Steiner and J.-M. Morvan.
\newblock Second fundamental measure of geometric sets and local approximation
  of curvatures.
\newblock {\em J. Differential Geom.}, 74(3):363--394, 2006.

\bibitem[DAL{\etalchar{+}}11]{farman}
J.~Digne, N.~Audfray, C.~Lartigue, C.~Mehdi-Souzani, and J.-M. Morel.
\newblock {Farman Institute 3D Point Sets - High Precision 3D Data Sets}.
\newblock {\em {Image Processing On Line}}, 1:281--291, 2011.

\bibitem[Del97]{delladio1997special}
S.~Delladio.
\newblock Special generalized gauss graphs and their application to
  minimization of functionals involving curvatures.
\newblock {\em Journal fur die Reine und Angewandte Mathematik}, pages 17--44,
  1997.

\bibitem[Del00]{delladio2000differential}
S.~Delladio.
\newblock Differential geometry of generalized submanifolds.
\newblock {\em Pacific Journal of Mathematics}, 194(2):285--313, 2000.

\bibitem[DPDRG16]{DePhilippis}
G.~De~Philippis, A.~De~Rosa, and F.~Ghiraldin.
\newblock Rectifiability of varifolds with locally bounded first variation with
  respect to anisotropic surface energies.
\newblock {\em Communications on Pure and Applied Mathematics}, 71, 09 2016.

\bibitem[DS95]{delladio1995oriented}
S.~Delladio and G.~Scianna.
\newblock Oriented and nonoriented curvature varifolds.
\newblock {\em Proceedings of the Royal Society of Edinburgh Section A:
  Mathematics}, 125(1):63--83, 1995.

\bibitem[Eig]{eigen}
Eigen.
\newblock A {C}++-template library for linear algebra.
\newblock \url{http://eigen.tuxfamily.org/}.

\bibitem[Fu93]{fu1993}
J.~H.~G. Fu.
\newblock Convergence of curvatures in secant approximations.
\newblock {\em J. Differential Geom.}, 37(1):177--190, 1993.

\bibitem[GMMM09]{giaquinta2009currents}
M.~Giaquinta, P.~M. Mariano, G.~Modica, and D.~Mucci.
\newblock Currents and curvature varifolds in continuum mechanics.
\newblock {\em Nonlinear partial differential equations and related topics.
  Dedicated to Nina U. Uraltseva on the occasion of her 75th birthday}, pages
  97--117, 2009.

\bibitem[HI01]{huisken2001inverse}
G.~Huisken and T.~Ilmanen.
\newblock The inverse mean curvature flow and the {R}iemannian penrose
  inequality.
\newblock {\em Journal of Differential Geometry}, 59(3):353--437, 2001.

\bibitem[Hut86a]{Hutchinson2}
J.~E. Hutchinson.
\newblock $c^{1,\alpha}$ multiple function regularity and tangent cone
  behaviour for varifolds with second fundamental form in ${L}^p$.
\newblock {\em Proceedings of Symposia in Pure Mathematics}, 44, 1986.

\bibitem[Hut86b]{Hutchinson}
J.~E. Hutchinson.
\newblock Second fundamental form for varifolds and the existence of surfaces
  minimising curvature.
\newblock {\em Indiana Univ. Math. J.}, 35(1), 1986.

\bibitem[Ilm93]{Ilmanen}
T.~Ilmanen.
\newblock Convergence of the {A}llen--{C}ahn equation to {B}rakke's motion by
  mean curvature.
\newblock {\em J. Differ. Geom.}, 38:417--461, 1993.

\bibitem[KMS14]{kuwert2014existence}
E.~Kuwert, A.~Mondino, and J.~Schygulla.
\newblock Existence of immersed spheres minimizing curvature functionals in
  compact 3-manifolds.
\newblock {\em Mathematische Annalen}, 359(1-2):379--425, 2014.

\bibitem[KT15]{TonegawaKim}
L.~Kim and Y.~Tonegawa.
\newblock On the mean curvature flow of grain boundaries.
\newblock {\em Annales de l'institut Fourier}, 67, 11 2015.

\bibitem[Lev98]{levin1998}
D.~Levin.
\newblock {The approximation power of moving least-squares}.
\newblock {\em Mathematics of Computation}, 67:1517--1531, 1998.

\bibitem[Luc08]{Luckhaus}
S.~Luckhaus.
\newblock Uniform rectifiability from mean curvature bounds.
\newblock {\em Z. Anal. Anwend.}, 27(4):451--461, 2008.

\bibitem[Man96]{Mantegazza}
C.~Mantegazza.
\newblock Curvature varifolds with boundary.
\newblock {\em J. Diff. Geom.}, 43:807--843, 1996.

\bibitem[Men13]{Menne2013}
U.~Menne.
\newblock Second order rectifiability of integral varifolds of locally bounded
  first variation.
\newblock {\em Journal of Geometric Analysis}, 23(2):709--763, Apr 2013.

\bibitem[Men16]{menne}
U.~Menne.
\newblock Weakly differentiable functions on varifolds.
\newblock {\em Indiana Univ. Math. J.}, 65(3):977--1088, 2016.

\bibitem[MN14]{CodaNeves}
F.~C. Marques and A.~Neves.
\newblock Min-max theory and the {W}illmore conjecture.
\newblock {\em Annals of Mathematics}, 179(2):683--782, 2014.

\bibitem[MOG11]{Merigot2009}
Q.~M{\'e}rigot, M.~Ovsjanikov, and L.~Guibas.
\newblock Voronoi-based curvature and feature estimation from point clouds.
\newblock {\em IEEE Transactions on Visualization and Computer Graphics},
  17:743--756, 2011.

\bibitem[Mon10]{Mondino}
A.~Mondino.
\newblock Existence of integral m-varifolds minimizing $\int |{A}|^p$ and $\int
  |{H}|^p$, $p>m$, in {R}iemannian manifolds.
\newblock {\em Calc. Var. PDE}, 2010.

\bibitem[Mor08]{morvan_book}
J.-M. Morvan.
\newblock {\em Generalized curvatures}, volume~2 of {\em Geometry and
  Computing}.
\newblock Springer-Verlag, Berlin, 2008.

\bibitem[MS18]{MenneSharrer2018}
U.~Menne and C.~Scharrer.
\newblock An isoperimetric inequality for diffused surfaces.
\newblock {\em Kodai Mathematical Journal}, 41(1):70--85, 2018.

\bibitem[Nan]{nanoflann}
Nanoflann.
\newblock A {C++} header-only library for nearest neighbor search with
  {KD}-trees.
\newblock \url{https://github.com/jlblancoc/nanoflann}.

\bibitem[NR17]{najman2017modern}
L.~Najman and P.~Romon.
\newblock {\em Modern approaches to discrete curvature}, volume 2184.
\newblock Springer, 2017.

\bibitem[Pit81]{Pitts}
J.~T. Pitts.
\newblock {\em Existence and Regularity of Minimal Surfaces on {R}iemannian
  Manifolds. (MN-27):}.
\newblock Princeton University Press, 1981.

\bibitem[PR14]{piccoli_rossi}
B.~Piccoli and F.~Rossi.
\newblock Generalized {W}asserstein distance and its application to transport
  equations with source.
\newblock {\em Arch. Ration. Mech. Anal.}, 211(1):335--358, 2014.

\bibitem[RS06]{Roger}
M.~R{\"o}ger and R.~Sch{\"a}tzle.
\newblock On a modified conjecture of {D}e {G}iorgi.
\newblock {\em Mathematische Zeitschrift}, 254(4):675--714, Dec 2006.

\bibitem[Sim83]{simon}
L.~Simon.
\newblock {\em Lectures on geometric measure theory}, volume~3 of {\em
  Proceedings of the Centre for Mathematical Analysis, Australian National
  University}.
\newblock Australian National University Centre for Mathematical Analysis,
  Canberra, 1983.

\bibitem[Spo93]{Spohn1993}
H.~Spohn.
\newblock Interface motion in models with stochastic dynamics.
\newblock {\em Journal of Statistical Physics}, 71(5):1081--1132, Jun 1993.

\bibitem[Ton03]{Tone}
Y.~Tonegawa.
\newblock Integrality of varifolds in the singular limit of reaction?diffusion
  equations.
\newblock {\em Hiroshima Math. J.}, 33:323--341, 2003.

\bibitem[TT16]{TakaTone}
K.~Takasao and Y.~Tonegawa.
\newblock Existence and regularity of mean curvature flow with transport term
  in higher dimensions.
\newblock {\em Math. Ann.}, 364:857--935, 2016.

\bibitem[Vil09]{villani_book}
C.~Villani.
\newblock {\em Optimal transport}, volume 338.
\newblock Springer-Verlag, Berlin, 2009.

\bibitem[Win82]{Wintgen1982}
P.~Wintgen.
\newblock Normal cycle and integral curvature for polyhedra in {R}iemannian
  manifolds.
\newblock {\em Differential Geometry. North-Holland Publishing Co.,
  Amsterdam-New York}, 1982.

\bibitem[YQ07]{Yang2007}
P.~Yang and X.~Qian.
\newblock Direct computing of surface curvatures for point-set surfaces.
\newblock {\em Eurographics Symposium on Point-Based Graphics}, pages 29--36,
  2007.

\bibitem[Z{\"a}h86]{zahle}
M.~Z{\"a}hle.
\newblock Integral and current representation of {F}ederer's curvature
  measures.
\newblock {\em Archiv der Mathematik}, 46(6):557--567, 1986.

\end{thebibliography}

\end{document}